\documentclass[10pt,a4paper]{amsart}
\usepackage{fullpage} 
\usepackage{amsmath}
\usepackage{amsthm}
\usepackage{amssymb}
\usepackage{mathtools}
\usepackage[only,mapsfrom]{stmaryrd}
\usepackage{graphicx}
\usepackage{cancel} 
\usepackage{color} 
\usepackage[all]{xy} 
\usepackage{tikz}
\usetikzlibrary{arrows,calc}
\usetikzlibrary{cd}
\usepackage{nicefrac}
\usepackage{enumitem}
\usepackage[colorlinks,citecolor=blue,linkcolor=blue,urlcolor=blue,filecolor=blue,breaklinks]{hyperref}
\usepackage{caption}
\usepackage{pifont}

\numberwithin{equation}{section}
\numberwithin{figure}{section}

\newtheorem{thm}{Theorem}[section]
\newtheorem{athm}{Theorem}

\newtheorem{acor}[athm]{Corollary}
\newtheorem{lem}[thm]{Lemma}
\newtheorem{prop}[thm]{Proposition}
\newtheorem{cor}[thm]{Corollary}

\newtheorem*{thm*}{Theorem}
\newtheorem*{conj*}{Conjecture}
\newtheorem*{cor*}{Corollary}
\newtheorem*{ques*}{Question}
\newtheorem*{claim*}{Claim}

\theoremstyle{definition}
\newtheorem{rem}[thm]{Remark}
\newtheorem{defn}[thm]{Definition}
\newtheorem{construction}[thm]{Construction}
\newtheorem{ex}[thm]{Example}
\newtheorem{notation}[thm]{Notation}
\newtheorem*{rem*}{Remark}

\newcommand{\cC}{\mathcal{C}}

\newcommand{\cP}{\mathcal{P}}

\newcommand{\cZ}{\mathcal{Z}}

\newcommand{\bC}{\mathbb{C}}
\newcommand{\bD}{\mathbb{D}}

\newcommand{\bN}{\mathbb{N}}

\newcommand{\bP}{\mathbb{P}}
\newcommand{\bQ}{\mathbb{Q}}
\newcommand{\bR}{\mathbb{R}}
\newcommand{\bS}{\mathbb{S}}

\newcommand{\bZ}{\mathbb{Z}}
\newcommand{\fC}{\mathfrak{C}}
\newcommand{\fF}{\mathfrak{F}}

\newcommand{\cmark}{\ding{51}}
\newcommand{\xmark}{\ding{55}}
\newcommand{\qmark}{\textbf{?}}

\newcommand{\Map}{\mathrm{Map}}
\newcommand{\PMap}{\mathrm{PMap}}
\newcommand{\Homeo}{\mathrm{Homeo}}
\newcommand{\Diff}{\mathrm{Diff}}
\newcommand{\Gr}{\mathrm{Gr}}
\newcommand{\Str}{\mathrm{Str}}
\newcommand{\Bij}{\mathrm{Bij}}
\newcommand{\supp}{\mathrm{supp}}
\newcommand{\image}{\mathrm{image}}
\newcommand{\Ends}{{\mathrm{Ends}}}
\newcommand{\longhookrightarrow}{\ensuremath{\lhook\joinrel\longrightarrow}}

\newcommand{\longtwoheadleftarrow}{\ensuremath{\twoheadleftarrow\joinrel\relbar}}
\newcommand{\incl}[3][right]%
{%
\draw[<-,>=#1 hook] #2 to ($ #2!0.5!#3 $);
\draw[->] ($ #2!0.5!#3 $) to #3;%
}
\newcommand{\inclusion}[5][right]%
{%
\draw[<-,>=#1 hook] #4 to ($ #4!0.5!#5 $) node[#2,font=\small]{#3};
\draw[->,>=stealth'] ($ #4!0.5!#5 $) to #5;%
}

\newcommand{\specialcell}[2][c]{%
  \begin{tabular}[#1]{@{}c@{}}#2\end{tabular}}

\renewcommand{\geq}{\geqslant}
\renewcommand{\leq}{\leqslant}

\title{Compact and finite-type support in the homology of big mapping class groups}

\author{Martin Palmer}
\address{Institutul de Matematică Simion Stoilow al Academiei Române, 21 Calea Griviței, 010702 Bucharest, Romania; School of Mathematics, University of Leeds, Leeds, LS2 9JT, UK}
\email{mpanghel@imar.ro; m.d.palmer-anghel@leeds.ac.uk}

\author{Xiaolei Wu}
\address{Shanghai Center for Mathematical Sciences, Jiangwan Campus, Fudan University, No.2005 Songhu Road, Shanghai, 200438, P.R. China}
\email{xiaoleiwu@fudan.edu.cn}

\date{17 July 2025}

\begin{document}

\begin{abstract}
For any infinite-type surface $S$, a natural question is whether the homology of its mapping class group contains any non-trivial classes that are supported on (i) a \emph{compact} subsurface or (ii) a \emph{finite-type} subsurface. Our purpose here is to study this question, in particular giving an almost-complete answer when the genus of $S$ is positive (including infinite) and a partial answer when the genus of $S$ is zero. Our methods involve the notion of \emph{shiftable subsurfaces} as well as homological stability for mapping class groups of finite-type surfaces.
\end{abstract}
\maketitle

\section*{Introduction}

In their seminal work \cite{MW07}, Madsen and Weiss calculated the stable
homology of the mapping class groups of compact, connected, orientable surfaces, in particular confirming the Mumford conjecture \cite{Mumford1983}. Let $L$ denote the \emph{Loch Ness monster surface}, the unique infinite-genus surface with one end and no boundary, and write $\Map_c(L)$ for the subgroup of the mapping class group $\Map(L) = \pi_0(\Homeo(L))$ of elements admitting compactly-supported representatives. Rationally, the Madsen--Weiss theorem has the following consequence.

\begin{thm*}[\cite{MW07}]
$H^\ast (\Map_c(L);\bQ) \cong \bQ[\kappa_1,\kappa_2,\ldots]$, where $\kappa_i$ is the Miller--Morita--Mumford class of degree $2i$.
\end{thm*}

Recently, much progress has been made towards calculating the homology of mapping class groups of infinite-type surfaces \cite{APV20,Domat2022,PalmerWu2,PalmerWu2024, MalesteinTao2024}. In particular, for the Loch Ness monster surface $L$, the authors showed in \cite[Proposition~5.3]{PalmerWu2} that $H_{\ast}(\Map(L);\bZ)$ is uncountable in every positive degree. The proof is constructive, but the (uncountably many) homology classes constructed do not have compact support. It is therefore natural to wonder whether $H_{\ast}(\Map(L);\bZ)$ contains \emph{any} (non-zero) classes with compact support, in other words, whether the map $H_\ast (\Map_c(L);\bZ) \to H_\ast (\Map(L);\bZ)$ induced by the inclusion $\Map_c(L) \subset \Map(L)$ has non-trivial image. In particular, does the dual class $\kappa_i^*$ of any Miller--Morita--Mumford class $\kappa_i$ survive in $H_\ast (\Map(L);\bQ)$?

\begin{athm}
\label{mainthm-LochNess}
For any field $K$, the map $H_\ast (\Map_c(L);K) \to H_\ast (\Map(L);K)$ is zero in positive degrees. In particular, for $K=\bQ$, all dual Miller--Morita--Mumford classes $\kappa_i^*$ are sent to zero in $H_\ast (\Map(L);\bQ)$.
\end{athm}

We do not know whether this result remains true if the field $K$ is replaced by $\bZ$ (see Remark \ref{rmk:from-fields-to-Z}).

\subsection*{The general questions.}
For any (connected, second-countable, orientable) infinite-type surface $S$ with $\partial S = \varnothing$, we study the following two questions about its mapping class group $\Map(S) = \pi_0(\Homeo(S))$:

\begin{ques*}
Does $\Map(S)$ contain non-zero classes in the image of $H_*(\Map(\Sigma)) \to H_*(\Map(S))$ for
    \begin{enumerate}[nosep,label=\textup{(\Roman*)}]
        \item\label{q-compact-pre} some compact subsurface $\Sigma \subset S$?
        \setcounter{enumi}{2}
        \item\label{q-finite-type-pre} some properly-embedded finite-type subsurface $\Sigma \subset S$?
    \end{enumerate}
\end{ques*}

The assumption in \ref{q-finite-type-pre} that $\Sigma \subset S$ is properly embedded is necessary for there to be a well-defined induced map $\Map(\Sigma) \to \Map(S)$ given by extending by the identity; see Lemma \ref{lem:extend-by-id}.

As the numbering suggests, there is in fact another intermediate question between \ref{q-compact-pre} and \ref{q-finite-type-pre}. To see this, we first discuss some subgroups of $\Map(S)$ as well as some colimit groups mapping into it.

\begin{defn}
\label{defn:restricting-support}
Let $\Map_c(S) \subseteq \Map(S) = \pi_0(\Homeo(S))$ denote the subgroup of mapping classes that may be represented by a homeomorphism $\varphi$ whose \emph{support} $\supp(\varphi) = \overline{\{ p \in S \mid \varphi(p) \neq p \}} \subseteq S$ is compact. Similarly, define $\Map_f(S) \subseteq \Map(S)$ to be the subgroup of mapping classes that may be represented by a homeomorphism $\varphi$ whose support is contained in a \emph{properly-embedded finite-type subsurface} of $S$, namely a subsurface of $S$ that is closed as a subset and whose fundamental group is finitely generated.
\end{defn}

\begin{defn}
Let us denote by $\fC(S) \subseteq \fF(S)$ the posets of compact subsurfaces of $S$ and of properly-embedded finite-type subsurfaces of $S$, ordered by inclusion. For a finite-type surface $\Sigma$, let us write $\PMap(\Sigma)$ for the subgroup of elements of $\Map(\Sigma)$ that fix the punctures of $\Sigma$ pointwise. (This is an index-$p!$ subgroup if $\Sigma$ has $p$ punctures.) Define
\begin{align*}
\Map_\fC(S) &:= \underset{\Sigma \in \fC(S)}{\mathrm{colim}}(\Map(\Sigma)) \\
\Map_\fF(S) &:= \underset{\Sigma \in \fF(S)}{\mathrm{colim}}(\Map(\Sigma)) \\
\PMap_\fF(S) &:= \underset{\Sigma \in \fF(S)}{\mathrm{colim}}(\PMap(\Sigma)).
\end{align*}
Notice that the analogous $\PMap_\fC(S)$ is simply $\Map_\fC(S)$ again, since compact surfaces have no punctures.
\end{defn}

There are natural homomorphisms
\begin{equation}
\label{eq:natural-maps}
\Map_\fC(S) \longrightarrow \PMap_\fF(S) \longrightarrow \Map_\fF(S) \longrightarrow \Map(S)
\end{equation}
induced by the inclusion of posets $\fC(S) \subseteq \fF(S)$, the inclusions $\PMap(\Sigma) \subseteq \Map(\Sigma)$ and the homomorphisms $\Map(\Sigma) \to \Map(S)$ given by extending homeomorphisms of $\Sigma$ by the identity on $S \smallsetminus \Sigma$.

Since homology commutes with colimits, Questions \ref{q-compact-pre} and \ref{q-finite-type-pre} above may be reformulated as follows, where we have added one intermediate question.

\begin{ques*}
Is there a non-zero element of $H_*(\Map(S))$ in the image of the map on homology induced by
    \begin{enumerate}[nosep,label=\textup{(\Roman*)}]
        \item\label{q-compact} $\Map_\fC(S) \to \Map(S)$?
        \item\label{q-finite-type-pure} $\PMap_\fF(S) \to \Map(S)$?
        \item\label{q-finite-type} $\Map_\fF(S) \to \Map(S)$?
    \end{enumerate}
\end{ques*}

Questions \ref{q-finite-type-pure} and \ref{q-finite-type} may be reformulated in terms of inclusions of subgroups of $\Map(S)$ as follows (see Definition \ref{defn:p} below for the notation $p_S$).

\begin{lem}[Lemma \ref{lem:colimits}]
\label{lem:reformulation}
The homomorphisms \eqref{eq:natural-maps} have the following properties:
    \begin{itemize}[nosep]
        \item $\PMap_\fF(S) \to \Map_\fF(S) \to \Map(S)$ are injective with images $\Map_c(S) \subseteq \Map_f(S) \subset \Map(S)$.
        \item $\Map_\fC(S) \to \PMap_\fF(S)$ is a central extension whose kernel is free abelian of rank $p_S$.
    \end{itemize}
In particular, Questions \ref{q-finite-type-pure} and \ref{q-finite-type} are equivalent to:
    \begin{enumerate}[nosep,label=\textup{(\Roman*)}]
        \setcounter{enumi}{1}
        \item Does the inclusion $\Map_c(S) \subset \Map(S)$ induce a non-zero map on homology?
        \item Does the inclusion $\Map_f(S) \subset \Map(S)$ induce a non-zero map on homology?
    \end{enumerate}
Moreover, if $p_S = 0$ then Questions \ref{q-compact} and \ref{q-finite-type-pure} are equivalent.
\end{lem}

Since the natural maps into $\Map(S)$ factor as in \eqref{eq:natural-maps}, we immediately observe:

\begin{rem}
\label{rmk:obvious-implications}
A positive answer to Question \ref{q-compact} implies a positive answer to Question \ref{q-finite-type-pure}, which implies a positive answer to Question \ref{q-finite-type}. However, a positive answer to Question \ref{q-finite-type-pure} does not necessarily imply a positive answer to Question \ref{q-compact}, as the surjective map $\Map_\fC(S) \to \PMap_\fF(S)$ does not necessarily induce surjective maps on homology in degrees greater than $1$.
\end{rem}

\begin{rem}
Theorem \ref{mainthm-LochNess} says that (with field coefficients) the answer to Question \ref{q-finite-type-pure} is negative for the surface $S=L$. It follows by Remark \ref{rmk:obvious-implications} that the answer to Question \ref{q-compact} is also negative for $S=L$. In fact, since $L$ has no punctures ($p_S = 0$), Questions \ref{q-compact} and \ref{q-finite-type-pure} are equivalent by Lemma \ref{lem:reformulation}.
\end{rem}

Our answers to Questions \ref{q-compact}--\ref{q-finite-type} depend on the \emph{genus} $g_S$ and the \emph{number of punctures} $p_S$ of $S$.

\begin{defn}[\emph{Punctures}]
\label{defn:p}
Consider the space $\Ends(S)$ of ends of $S$, together with its closed subspace $\Ends_{np}(S)$ of non-planar ends. A \emph{puncture} of $S$ is an isolated point of the space $\Ends(S) \smallsetminus \Ends_{np}(S)$; in other words, it is an end of $S$ that is not accumulated by genus and is not a limit points of other ends of $S$. Denote the set of punctures by $\cP(S)$. Since the space $\Ends(S) \smallsetminus \Ends_{np}(S)$ is separable, this set is at most countable and we write $p_S \in \{0,1,2,3,\ldots,\infty\}$ for its cardinality.
\end{defn}

\begin{notation}
For integers $g,n,b \geq 0$, we write $\Sigma_{g,b}^n$ for the unique connected, finite-type, orientable surface of genus $g$ with $b$ boundary components and $n$ punctures. If $n = 0$ we elide it from the notation, and similarly for $b$.
\end{notation}

\begin{defn}[\emph{Genus}]
\label{defn:g}
Let $S$ be any surface. Its \emph{genus} $g_S$ is the maximum integer $g\geq 0$ for which there is an embedding $\Sigma_{g,1} \hookrightarrow S$, if there is such a maximum. Otherwise, we set $g_S=\infty$. 
\end{defn}

\begin{thm*}[Theorems \ref{mainthm-infinite-genus}--\ref{mainthm-genus-zero-infinite-punctures-2}]
Let $S$ be any connected, second countable, infinite-type surface with $\partial S = \varnothing$. Answers to Questions \ref{q-compact}--\ref{q-finite-type} for $S$ are given in Table \ref{tab:results} on page \pageref{tab:results}.
\end{thm*}

{\def\arraystretch{1.4}
\begin{table}[t]
\captionsetup{singlelinecheck=off}
    \centering
    \begin{tabular}{|r|c|c|c|c|c|}
        \hline
        & $p_S = \infty$ & $4 \leq p_S < \infty$ & $p_S \in \{2,3\}$ & $p_S = 1$ & $p_S = 0$ \\
        \hline
        $g_S = \infty$ & \specialcell{$\exists \text{ mixed end}$: \xmark\xmark\xmark \\ $\nexists \text{ mixed end}$: \xmark\qmark\qmark} & \multicolumn{3}{c|}{\xmark\cmark\cmark} & \xmark\xmark\xmark \\
        \hline
        $0<g_S<\infty$ & \multicolumn{5}{c|}{\cmark\cmark\cmark} \\
        \hline
        $g_S = 0$ & \specialcell{$\Ends(S)$ is $\text{TD}_{\geq 4}$: \cmark\cmark\cmark \\ $\Ends(S) \cong [0,\omega^\alpha]$: \xmark\xmark\xmark \\ otherwise: \qmark\qmark\qmark} & \cmark\cmark\cmark & \qmark\qmark\cmark & \multicolumn{2}{c|}{\xmark\xmark\xmark} \\
        \hline
    \end{tabular}
    \caption{Summary of the results of Theorems \ref{mainthm-infinite-genus}--\ref{mainthm-genus-zero-infinite-punctures-2} answering Questions \ref{q-compact}--\ref{q-finite-type}. See Notation \ref{notation-for-table} for terminology.}
    \label{tab:results}
\end{table}
}

\begin{notation}
\label{notation-for-table}
In Table \ref{tab:results}, a triple $\text{\textbf{ABC}}$ with $\text{\textbf{A}},\text{\textbf{B}},\text{\textbf{C}} \in \{ \text{\cmark} , \text{\xmark} , \text{\qmark} \}$ encodes the answers to Questions \ref{q-compact}, \ref{q-finite-type-pure}, \ref{q-finite-type} in that order. The answer to Question \ref{q-compact} is positive if $\text{\textbf{A}} = \text{\cmark}$, negative if $\text{\textbf{A}} = \text{\xmark}$ and unknown (to us) if $\text{\textbf{A}} = \text{\qmark}$, and similarly for Questions \ref{q-finite-type-pure} and \ref{q-finite-type} with $\text{\textbf{A}}$ replaced by $\text{\textbf{B}}$ and $\text{\textbf{C}}$ respectively.
One caveat is that (almost) all negative answers assume field coefficients for homology, whereas all positive answers assume integral coefficients for homology; see Theorems \ref{mainthm-infinite-genus}--\ref{mainthm-genus-zero-infinite-punctures-2} for the precise statements.
The other notation in Table \ref{tab:results} is explained in Definitions \ref{defn:p}, \ref{defn:g}, \ref{def:mixed-end}, \ref{def:TDn} and Notation \ref{notation:ordinal-interval}. In the rest of this paper, we will sometimes write simply $H_*(-)$ to refer to integral homology $H_*(-;\bZ)$.
\end{notation}

In the remainder of the introduction, we explain the results summarised in Table \ref{tab:results} in more detail.

\subsection*{Results in infinite genus.}

We begin with the infinite-genus ($g_S = \infty$) setting, for which we need one preliminary definition. If both $g_S$ and $p_S$ are infinite, then $S$ must have at least one end that is \emph{accumulated by genus} (every neighbourhood of the end has infinite genus) and at least one end that is \emph{accumulated by punctures} (every neighbourhood of the end has infinitely many punctures).

\begin{defn}[\emph{Mixed end}]
\label{def:mixed-end}
We say that $S$ has a \emph{mixed end} if it has an end that is accumulated by both genus and punctures.
\end{defn}

Having a mixed end implies, of course, that $g_S = p_S = \infty$. The converse is not true, however: if we remove from the Loch Ness monster surface a subset homeomorphic to $\bN^+$, the one-point compactification of $\bN$, then the resulting surface has $g_S = p_S = \infty$ but no mixed ends.

Generalising Theorem \ref{mainthm-LochNess} for the Loch Ness monster surface, we have the following result.

\begin{athm}
\label{mainthm-infinite-genus}
Suppose that $g_S = \infty$.
\begin{enumerate}[nosep,label=\textup{(\arabic*)}]
\item\label{mainthm-infinite-genus-1} The answer to Question \ref{q-compact} is negative for homology with any field coefficients.
\end{enumerate}
For Questions \ref{q-finite-type-pure} and \ref{q-finite-type}:
\begin{enumerate}[nosep,label=\textup{(\arabic*)}]
\setcounter{enumi}{1}
\item\label{mainthm-infinite-genus-2} If $p_S = 0$ then $\Map_f(S) \subset \Map(S)$ induces the zero map on homology with field coefficients.
\item\label{mainthm-infinite-genus-3} If $0 < p_S < \infty$ then $\Map_c(S) \subset \Map(S)$ induces a non-zero map on integral homology.
\item\label{mainthm-infinite-genus-4} If $p_S = \infty$ and $S$ has a mixed end, then $\Map_f(S) \subset \Map(S)$ induces the zero map on homology with field coefficients.
\end{enumerate}
\end{athm}

In the context of Questions \ref{q-finite-type-pure} and \ref{q-finite-type}, our methods do not apply if $g_S = p_S = \infty$ but $S$ does \emph{not} have a mixed end, so in this case Questions \ref{q-finite-type-pure} and \ref{q-finite-type} remain open.

\begin{rem}
In case \ref{mainthm-infinite-genus-3} of Theorem \ref{mainthm-infinite-genus}, we prove something stronger than simply the statement that the induced map $H_*(\Map_c(S);\bZ) \to H_*(\Map(S);\bZ)$ is non-zero: its image contains a $\bZ$ summand in every even degree; see Proposition \ref{prop:finite-p-part-2}.
\end{rem}

\begin{rem}
\label{rmk:from-fields-to-Z}
In the cases where we prove, in Theorem \ref{mainthm-infinite-genus}, that a group homomorphism induces the zero map on homology with all field coefficients, it does \emph{not} automatically follow that the same statement is also true with integral coefficients. Indeed it is possible in general for homomorphisms $G \to H$ to induce trivial maps on homology with all field coefficients but not with integral coefficients. An example is given by any non-trivial homomorphism $\bZ \to \bQ/\bZ$: it is non-trivial on $H_1(-;\bZ)$ by construction, but trivial on homology with field coefficients because $\bQ/\bZ \otimes_\bZ K = 0$ for any field $K$. See also Remark \ref{rmk:field-vs-Z} for why we require field coefficients in the proof.
\end{rem}

\subsection*{Results in finite positive genus.}

In the case when $S$ has finite but positive genus, the answers to Questions \ref{q-compact}--\ref{q-finite-type} are easy to state.

\begin{athm}
\label{mainthm-finite-positive-genus}
Suppose that $0 < g_S < \infty$. Then the integral homology $H_*(\Map(S);\bZ)$ contains non-zero classes that are supported on $\Map(\Sigma)$ for compact $\Sigma \subset S$. In other words, with integral coefficients, the answer to Question \ref{q-compact} is positive; hence the answers to Questions \ref{q-finite-type-pure} and \ref{q-finite-type} are also positive.
\end{athm}

\subsection*{Results in genus zero.}

When $S$ has genus zero, its homeomorphism type is completely determined by its space of ends $\Ends(S)$, which may be any space that is homeomorphic to a closed subset of the Cantor set $\cC$ (see \S\ref{s:surfaces} for more details). In this case \emph{punctures} of $S$ are simply isolated points of $\Ends(S)$. If the set $\cP(S)$ of punctures is finite, then $\Ends(S)$ is homeomorphic to the topological disjoint union $\cC \sqcup \cP(S)$, where $\cP(S)$ has the discrete topology (see \S\ref{ss:structure-of-end-spaces}). There are therefore two cases:
\begin{itemize}[nosep]
\item[\textbf{1.}] $\Ends(S)$ is homeomorphic to $\cC \sqcup \{1,\ldots,p\}$ for some non-negative integer $p = p_S < \infty$;
\item[\textbf{2.}] $\Ends(S)$ has (countably) infinitely many isolated points, i.e.~$p_S = \infty$.
\end{itemize}

\subsubsection*{Case \textup{\textbf{1.}}}
In the first case (finitely many punctures) we have the following.

\begin{athm}
\label{mainthm-genus-zero-finite-punctures}
Suppose that $g_S = 0$ and $0\leq p_S < \infty$. Then we have:
\begin{enumerate}[nosep,label=\textup{(\arabic*)}]
\item\label{mainthm-genus-zero-finite-punctures-1} If $p_S \in \{0,1\}$ then $\Map_f(S) \subset \Map(S)$ induces the zero map on homology with any coefficients.
\item\label{mainthm-genus-zero-finite-punctures-2} If $p_S \geq 2$ then $\Map_f(S) \subset \Map(S)$ induces a non-zero map on homology with integral coefficients.
\item\label{mainthm-genus-zero-finite-punctures-3} If $p_S \geq 4$ then $H_*(\Map(S);\bZ)$ contains non-zero classes supported on a compact $\Sigma \subset S$.
\end{enumerate}
\end{athm}

In short, using the terminology of Notation \ref{notation-for-table} and the implications of Remark \ref{rmk:obvious-implications}, the answers to Questions \ref{q-compact}--\ref{q-finite-type} in the three cases of Theorem \ref{mainthm-genus-zero-finite-punctures} are \xmark\xmark\xmark, \qmark\qmark\cmark and \cmark\cmark\cmark respectively. The two settings not covered by Theorem \ref{mainthm-genus-zero-finite-punctures} are Questions \ref{q-compact} and \ref{q-finite-type-pure} when $g_S = 0$ and $p_S \in \{2,3\}$.

\subsubsection*{Case \textup{\textbf{2.}}}
In the second case (infinitely many punctures), our results are much more partial, and the answers to Questions \ref{q-compact}--\ref{q-finite-type} appear to depend very subtly on the structure of $\Ends(S)$, which may be very complicated (in particular, there are uncountably many different homeomorphism types that $\Ends(S)$ may have in the case $p_S = \infty$). To state our results, we need some preliminary definitions and recollections.

\begin{defn}
\label{def:topologically-distinguished-subset}
A subset $A$ of a space $X$ is \emph{topologically distinguished} if one can detect whether a point $x \in X$ lies in $A$ by looking at an arbitrarily small neighbourhood of $x$ in $X$. Formally, this means that if $a \in A$ and $x \in X \smallsetminus A$ and $U,V$ are neighbourhoods of $a,x$ in $X$ respectively, then the based spaces $(U,a)$ and $(V,x)$ are \emph{not} homeomorphic.
\end{defn}

\begin{notation}
\label{notation:ordinal-interval}
Write $\omega$ for the first infinite ordinal (the ordinal of $\bN$) and denote by $[0,\beta]$ the \emph{closed ordinal interval} below $\beta$, i.e.~the ordinal $\beta + 1$ given the order topology. See \S\ref{ss:countable-end-spaces} for more details.
\end{notation}

The space $\Ends(S)$ is compact and Hausdorff, so if it is in addition countable (and non-empty), then it must be homeomorphic to the disjoint union of $n$ copies of $[0,\omega^\alpha]$ for a (unique) positive integer $n$ and countable ordinal $\alpha$. This is a theorem of Mazurkiewicz and Sierpi{\'n}ski \cite{MazurkiewiczSierpinski1920}, recalled as Theorem \ref{thm:MazurkiewiczSierpinski} in \S\ref{ss:countable-end-spaces} below.

\begin{notation}
For a positive integer $n$ and countable ordinal $\alpha$, write $O(n,\alpha)$ for the topological disjoint union of $n$ copies of the space $[0,\omega^\alpha]$.
\end{notation}

The discussion above implies that, if $\Ends(S)$ is countable and non-empty, then it is homeomorphic to $O(n,\alpha)$ for a unique pair $(n,\alpha)$.

\begin{defn}
\label{def:TDn}
For an integer $n\geq 0$, we say that a space $X$ is $\text{TD}_{\geq n}$ if it has a finite, topologically distinguished subset $A \subseteq X$ of cardinality at least $n$.
\end{defn}

\begin{ex}
\label{eg:TDn}
For example, the maximal element $\omega^\alpha \in [0,\omega^\alpha]$ is topologically distinguished (it is the unique point of Cantor-Bendixson rank $\alpha + 1$), so it follows that $O(n,\alpha)$ is $\text{TD}_{\geq m}$ for any $m\leq n$.
\end{ex}

Our first result in the setting $(g_S,p_S) = (0,\infty)$ is the following, in which the end-space $\Ends(S)$ may be either countable or uncountable.

\begin{athm}
\label{mainthm-genus-zero-infinite-punctures-1}
Suppose that $g_S = 0$ and that $\Ends(S)$ is $\text{TD}_{\geq 4}$. Then $H_*(\Map(S);\bZ)$ contains non-zero classes supported on a compact $\Sigma \subset S$.
\end{athm}

If $\Ends(S)$ is uncountable (and $g_S = 0$) we do not have any further answers to Questions \ref{q-compact}--\ref{q-finite-type}. However, if $\Ends(S)$ is countable -- and is therefore homeomorphic to $O(n,\alpha)$ for some $n$ and $\alpha$ by the discussion above -- we may go further.
Let us therefore assume that $g_S = 0$ and $\Ends(S) \cong O(n,\alpha)$ for a positive integer $n$ and countable ordinal $\alpha$. We first observe that, if $n\geq 4$, Questions \ref{q-compact}--\ref{q-finite-type} are all answered positively by Theorem \ref{mainthm-genus-zero-infinite-punctures-1}, since $O(n,\alpha)$ is $\text{TD}_{\geq 4}$ by Example \ref{eg:TDn}. It therefore remains to consider $n \in \{1,2,3\}$. Our second result in the setting $(g_S,p_S) = (0,\infty)$ provides the (opposite) answer in the case $n=1$.

\begin{athm}
\label{mainthm-genus-zero-infinite-punctures-2}
Suppose that $g_S = 0$ and that $\Ends(S) \cong O(1,\alpha) = [0,\omega^\alpha]$. Then $\Map_f(S) \subset \Map(S)$ induces the zero map on homology with any field coefficients.
\end{athm}

The special case when $\alpha = 1$ corresponds to the \emph{flute surface}, which is the plane minus a countable discrete subset (for example it may be modelled concretely as $\bR^2 \smallsetminus \bZ^2$). Theorem \ref{mainthm-genus-zero-infinite-punctures-2} therefore includes the following special case, which we highlight as a corollary.

\begin{acor}
\label{maincor-flute-surface}
For any field $K$, the homology $H_*(\Map(\bR^2 \smallsetminus \bZ^2);K)$ does not contain any non-zero classes that admit compact support, or even support of finite type.
\end{acor}

By contrast, we note that the (integral) homology of $\Map(\bR^2 \smallsetminus \bZ^2)$ is very large: it is uncountable in every positive degree, by \cite[Theorem B]{PalmerWu2}. More generally, \cite[Theorem B]{PalmerWu2} implies the same statement about the integral homology of $\Map(S)$ whenever $g_S = 0$ and $\Ends(S) \cong O(1,\alpha)$ for a countable \emph{successor} ordinal $\alpha$. (Whether $\alpha$ is a successor or a limit ordinal is an important qualitative difference in the topology of $S$, and indeed the proof of Theorem \ref{mainthm-genus-zero-infinite-punctures-2} is different in these two cases.)

The remaining cases (in the setting $(g_S,p_S) = (0,\infty)$ and for countable $\Ends(S)$) are $n \in \{2,3\}$. For these two cases, we believe that the case $n=2$ will behave as in Theorem \ref{mainthm-genus-zero-infinite-punctures-2} whereas the case $n=3$ will behave as in Theorem \ref{mainthm-genus-zero-infinite-punctures-1}.

\subsection*{Outline.}

After recollections about infinite-type surfaces and their end-spaces in \S\ref{s:surfaces}, the organisation of the proofs of Theorems \ref{mainthm-infinite-genus}--\ref{mainthm-genus-zero-infinite-punctures-2} is explained in \S\ref{s:outline}. We prove our vanishing results in \S\ref{s:shiftable}--\S\ref{s:Cantor-tree-punctured}, with the core argument in most cases being Proposition \ref{prop:grid-surface} in \S\ref{s:shiftable}, and we prove our non-vanishing results in \S\ref{s:non-trivial-classes}--\S\ref{s:BT-construction}.

\subsection*{Acknowledgements.}

MP was partially supported by a grant of the Romanian Ministry of Education and Research, CNCS - UEFISCDI, project number PN-III-P4-ID-PCE-2020-2798, within PNCDI III, as well as by grants of the Ministry of Research, Innovation and Digitization, CNCS - UEFISCDI, project numbers PN-IV-P1-PCE-2023-2001 and PN-IV-P2-2.1-TE-2023-2040, within PNCDI IV.

XW is currently a member of LMNS and supported by NSFC 12326601. Part of this work was done when he was visiting ICMAT and he thanks Javier Aramayona and his group for the warm hospitality. He also thanks Francesco Fournier-Facio for discussions related to acyclic spaces and suggesting that Proposition \ref{prop:grid-surface} might be proved using \cite[Proposition 4.1]{BaumslagDyerHeller1980}, which led to Remark \ref{rem-pseudo-mitosis-pf}.

We thank Oscar Randal-Williams for pointing out a way to simplify our arguments in \S\ref{s:Cantor-tree-punctured}. We thank N\'estor Col\'in Hern\'andez, George Raptis and Rita Jim\'enez Rolland for each independently asking us the question of whether the dual Miller--Morita--Mumford classes vanish in the homology of the mapping class group of the Loch Ness monster surface. By Theorem \ref{mainthm-LochNess}, the answer to their question is \emph{yes}. Finally, we thank the anonymous referee for helpful suggestions and corrections.

\section{Infinite-type surfaces and their end-spaces}
\label{s:surfaces}

\subsection{Surfaces.}
\label{ss:classification}

Throughout this paper, all surfaces are assumed to be second countable, connected, orientable and to have compact boundary. A surface $S$ has \emph{finite type} if its fundamental group is finitely generated; otherwise it has \emph{infinite type}. The classification of surfaces is due to von Ker\'ekj\'art\'o \cite{vKer23} and Richards \cite{Ri63}, and crucially involves the \emph{end-space} $\Ends(S)$ of a surface $S$, which is by definition the boundary of the \emph{Freudenthal compactification} $\overline{S}$ of $S$ (see for example \cite[\S 2.1]{PalmerWu2} for more details) and is always homeomorphic to a closed subset of the Cantor set $\cC$. An end of $S$ is \emph{planar} if it has a neighbourhood in $\overline{S}$ that embeds into the plane; otherwise it is \emph{non-planar}. The (closed) subspace of non-planar ends is denoted by $\Ends_{np}(S) \subseteq \Ends(S)$.

\begin{thm}[{\cite[Theorems 1 and 2]{Ri63}}]
\label{thm:clas-inf-sur}
Let $S_1,S_2$ be two surfaces of genera $g_1,g_2 \in \bN \cup \{\infty\}$ with $b_1,b_2 \in \bN$ boundary components respectively. They are homeomorphic if and only if $g_1 = g_2$, $b_1 = b_2$ and there is a homeomorphism of pairs of spaces
\[
(\Ends(S_1), \Ends_{np}(S_1)) \;\cong\; (\Ends(S_2), \Ends_{np}(S_2)).
\]
Conversely, given any tuple $(g,b,Y,X)$, where $g \in \bN \cup \{\infty\}$, $b \in \bN$ and $X \subseteq Y \subseteq \cC$ is a nested pair of closed subsets of the Cantor set $\cC$, subject to the condition that $g = \infty$ if and only if $X \neq \varnothing$, there exists a surface $S$ of genus $g$ with $b$ boundary components such that $(\Ends(S),\Ends_{np}(S)) \cong (Y,X)$.
\end{thm}

\subsection{End-spaces.}
\label{ss:structure-of-end-spaces}

By Theorem \ref{thm:clas-inf-sur}, the possible end-spaces of surfaces are precisely the closed subsets of the Cantor set $\cC$; this motivates the following terminology.

\begin{defn}
A space $X$ is an \emph{end-space} if it is homeomorphic to a closed subset of the Cantor set $\cC$.
\end{defn}

An important result about the structure of end-spaces is the Cantor-Bendixson theorem, which we recall next.

\begin{defn}
Let $X$ be any space. The \emph{Cantor-Bendixson filtration} of $X$ is the transfinite descending filtration $X_\alpha$ of $X$ defined by $X_0 = X$, $X_{\alpha + 1}$ is obtained from $X_\alpha$ by discarding all isolated points (in other words $X_{\alpha + 1}$ is the \emph{derived set} of $X_\alpha$) and $X_\lambda = \bigcap_{\alpha < \lambda} X_\alpha$ for limit ordinals $\lambda$. For cardinality reasons, there is always some $\alpha$ such that $X_{\alpha + 1} = X_\alpha$, in other words $X_\alpha$ has no isolated points. The \emph{Cantor-Bendixson rank} $\lvert X \rvert_{\mathrm{CB}}$ of a space $X$ is the smallest $\alpha$ such that $X_{\alpha + 1} = X_\alpha$; this subspace is called the \emph{perfect kernel} $\kappa(X)$ of $X$. For a point $x \in X \smallsetminus \kappa(X)$, its \emph{Cantor-Bendixson rank} $\mathrm{CB}_X(x)$ is the smallest $\alpha$ for which $x \notin X_\alpha$. Thus we have $\lvert X \rvert_{\mathrm{CB}} = \mathrm{sup}\{ \mathrm{CB}_X(x) \mid x \in X \smallsetminus \kappa(X) \}$.
\end{defn}

\begin{thm}[Cantor-Bendixson]
\label{thm:CB}
If $X$ is a Polish space, i.e.~it is separable and completely metrisable, then its Cantor-Bendixson rank $\lvert X \rvert_{\mathrm{CB}}$ is countable.
\end{thm}

This applies in particular to all end-spaces, since they are Polish spaces. An immediate corollary is the following important structurul result about uncountable end-spaces.

\begin{cor}
\label{cor:uncountable-end-spaces}
Every uncountable end-space $X \subseteq \cC$ has a subspace homeomorphic to $\cC$ whose complement in $X$ is countable.
\end{cor}
\begin{proof}
At each step of the Cantor-Bendixson filtration of $X$ only countably many points are removed, so Theorem \ref{thm:CB} implies that $X \smallsetminus \kappa(X)$ is countable. Thus $\kappa(X)$ is non-empty, since $X$ is uncountable. So $\kappa(X)$ is a non-empty perfect subspace of $\cC$, which implies that it is homeomorphic to $\cC$.
\end{proof}

In particular, if $X$ is uncountable and has only finitely many isolated points, it is homeomorphic to $\cC \sqcup \{1,\ldots,p\}$ for some non-negative integer $p$.

\subsection{Countable end-spaces and ordinal intervals.}
\label{ss:countable-end-spaces}

Despite the structural result of Corollary \ref{cor:uncountable-end-spaces}, the structure of uncountable end-spaces may still be very complicated. In contrast, countable end-spaces are completely classified. It follows directly from the definitions that \emph{countable end-spaces} are the same as \emph{countable compact Hausdorff spaces}, and the latter are classified in terms of certain \emph{ordinal spaces}. We refer to \cite{Sierpinski1958} or \cite{Jech2003} for the basic notions of ordinals and ordinal arithmetic.

\begin{defn}
For an ordinal $\alpha$, the \emph{closed ordinal interval} $[0,\alpha]$ is the ordinal $\alpha + 1 = \{0,1,2,\ldots,\alpha\}$ equipped with the order topology. For an ordinal $\alpha$ and positive integer $n$, we write $O(n,\alpha)$ for the ordinal interval $[0,\omega^\alpha.n]$, equivalently the disjoint union of $n$ copies of the ordinal interval $[0,\omega^\alpha]$.
\end{defn}

\begin{rem}
\label{rmk:Onalpha}
The spaces $O(n,\alpha)$ are pairwise non-homeomorphic: they may be distinguished by the property that $O(n,\alpha)$ has exactly $n$ points of Cantor-Bendixson rank $\alpha + 1$ and no points of higher Cantor-Bendixson rank (so its Cantor-Bendixson rank as a space is also equal to $\alpha + 1$).
\end{rem}

Closed ordinal intervals are compact and Hausdorff. Conversely, we have:

\begin{thm}[\cite{MazurkiewiczSierpinski1920}]
\label{thm:MazurkiewiczSierpinski}
Every countable compact Hausdorff space is homeomorphic to $O(n,\alpha)$ for some (necessarily unique) positive integer $n$ and countable ordinal $\alpha$.
\end{thm}

\begin{ex}
Any ordinal $\alpha$ has a unique \emph{Cantor normal form} $\alpha = \omega^{\beta_1}.n_1 + \cdots + \omega^{\beta_k}.n_k$ for positive integers $n_1,\ldots,n_k$ and ordinals $\beta_1 > \cdots > \beta_k$. In this case we have $[0,\alpha] \cong O(n_1,\beta_1)$.
\end{ex}

This classification, together with the Cantor-Bendixson filtration, may be used to calculate the results of various operations on closed ordinal intervals. We record here several of these that will be used later.

\begin{lem}
\label{lem:ordinal-interval-calculations}
We have the following identifications, where all ordinals are assumed to be countable.
\begin{itemize}[nosep]
\item Let $\alpha_1,\ldots,\alpha_n$ be a finite sequence of ordinals with unique maximum $\alpha_1$. Then the disjoint union $[0,\omega^{\alpha_1}] \sqcup \cdots \sqcup [0,\omega^{\alpha_n}]$ is homeomorphic to $[0,\omega^{\alpha_1}]$.
\item The one-point compactification of the disjoint union of countably infinitely many copies of $[0,\omega^\alpha]$ is homeomorphic to $[0,\omega^{\alpha+1}]$.
\item Let $\lambda$ be a limit ordinal and let $(\alpha_\beta)_{\beta < \delta}$ be a $\delta$-indexed sequence of smaller ordinals, for another ordinal $\delta$, whose supremum is $\lambda$. Then the one-point compactification of the disjoint union over all $\beta < \delta$ of $[0,\omega^{\alpha_\beta}]$ is homeomorphic to $[0,\omega^\lambda]$.
\end{itemize}
\end{lem}

\begin{rem}
\label{rmk:cofinality}
Recall that the \emph{cofinality} of an ordinal $\alpha$ is the smallest ordinal $\delta$ that admits a strictly increasing map $\delta \to \alpha$ whose image is cofinal. If $\lambda$ is a countable limit ordinal, its cofinality is $\omega = \bN$, so in this case there always exists an ordinary ($\bN$-indexed) sequence as in the third point of Lemma~\ref{lem:ordinal-interval-calculations}. We note however that the third point of Lemma \ref{lem:ordinal-interval-calculations} does not require the sequence to be strictly increasing.
\end{rem}

\begin{proof}[Proof of Lemma \ref{lem:ordinal-interval-calculations}]
In each case, the space under consideration is evidently compact, Hausdorff and countable; we shall study its Cantor-Bendixson filtration and then apply Theorem \ref{thm:MazurkiewiczSierpinski}. In the first case, since $\alpha_1$ is the unique maximum of $\alpha_1,\ldots,\alpha_n$, the $\alpha_1$-th term of the Cantor-Bendixson filtration is the single point $\omega^{\alpha_1} \in [0,\omega^{\alpha_1}]$. Thus the result follows from Theorem \ref{thm:MazurkiewiczSierpinski} and the characterisation of the spaces $O(n,\alpha)$ in Remark \ref{rmk:Onalpha}.

In the second case, the $\alpha$-th term of the Cantor-Bendixson filtration is the disjoint union of countably infinitely many copies of $\{\omega^\alpha\}$ together with the point at infinity. The point at infinity is therefore the unique point of (maximal) Cantor-Bendixson rank $\alpha + 2$.

In the third case, each component $[0,\omega^{\alpha_\beta}]$ of the disjoint union vanishes before the $\lambda$-th term of the Cantor-Bendixson filtration, since $\lambda > \alpha_\beta$. It will therefore suffice to prove that the point at infinity of the one-point compactification \emph{does} lie in the $\lambda$-th term of the Cantor-Bendixson filtration, since it will then follow that it is the unique point of (maximal) Cantor-Bendixson rank $\lambda + 1$. Suppose for a contradiction that the point at infinity of the one-point compactification does \emph{not} lie in the $\lambda$-th term of the Cantor-Bendixson filtration; it must therefore vanish when passing from the $\gamma$-th term to the $(\gamma + 1)$-st term of the Cantor-Bendixson filtration, for some $\gamma < \lambda$. This means that it is an isolated point in the $\gamma$-th term of the Cantor-Bendixson filtration. By definition of the one-point compactification, this can only occur if the space that it is compactifying is already compact, which means that all but finitely many of the components $[0,\omega^{\alpha_\beta}]$ of the disjoint union must have vanished already by the $\gamma$-th term of the Cantor-Bendixson filtration. However, the component $[0,\omega^{\alpha_\beta}]$ vanishes precisely at the $(\alpha_\beta + 1)$-st term, so this means that all but finitely many of the $\alpha_\beta$ are smaller than $\gamma$. But this contradicts the assumption that $\lambda$ is the supremum of the $\alpha_\beta$.
\end{proof}

\section{Preliminaries on finite-type and compact support; organisation of the proofs}
\label{s:outline}

For definiteness, let us first recall the definition of the mapping class group of a surface, as well as a basic construction that says essentially that it is functorial with respect to proper inclusions of surfaces.

\begin{defn}
\label{def:mcg}
For a surface $S$, its \emph{mapping class group} is $\Map(S) = \pi_0(\Homeo_\partial(S))$, the group of isotopy classes of homeomorphisms of $S$ that restrict to the identity on $\partial S$.
\end{defn}

\begin{lem}
\label{lem:extend-by-id}
If $\Sigma \subseteq S$ is a properly-embedded subsurface, there is a well-defined homomorphism
\begin{equation}
\label{eq:extend-by-id}
\iota \colon \Map(\Sigma) \longrightarrow \Map(S)
\end{equation}
given by extending homeomorphisms of $\Sigma$ by the identity on $\Map(S)$.
\end{lem}
\begin{proof}
To see that this is well-defined one just has to check that any homeomorphism representing an element of $\Map(\Sigma)$ is the identity on its topological boundary as a subset of $S$, which is $\Sigma \cap (\overline{S \smallsetminus \Sigma})$. The assumption that $\Sigma \subseteq S$ is a subsurface that is properly embedded -- equivalently: closed as a subset of $S$ -- implies that $\Sigma \cap (\overline{S \smallsetminus \Sigma})$ is contained in $\partial\Sigma$, the boundary of $\Sigma$ as an abstract surface. But by Definition~\ref{def:mcg}, homeomorphisms representing elements of $\Map(\Sigma)$ restrict to the identity on $\partial\Sigma$, hence in particular on $\Sigma \cap (\overline{S \smallsetminus \Sigma})$.
\end{proof}

Our first goal in this section is to prove Lemma \ref{lem:reformulation}, which we recall here.

\begin{lem}[Lemma~\ref{lem:reformulation}]
\label{lem:colimits}
The homomorphisms \eqref{eq:natural-maps} have the following properties:
    \begin{itemize}[nosep]
        \item $\PMap_\fF(S) \to \Map_\fF(S) \to \Map(S)$ are injective with images $\Map_c(S) \subseteq \Map_f(S) \subset \Map(S)$.
        \item $\Map_\fC(S) \to \PMap_\fF(S)$ is a central extension whose kernel is free abelian of rank $p_S$.
    \end{itemize}
\end{lem}

(The remaining statements of Lemma \ref{lem:reformulation} follow immediately from these ones.) We will deduce this lemma from the following fact, which is well-known in the finite-type setting and generalises with no change to the infinite-type setting.

\begin{prop}
\label{prop:injectivity}
Let $S$ be an infinite-type surface with $\partial S = \varnothing$ and $\Sigma \subset S$ a properly-embedded finite-type subsurface. Assume that $\Sigma$ is not an annulus and that $S$ is obtained from $\Sigma$ by attaching $S_1,\ldots,S_b$ along the boundary-components $C_1,\ldots,C_b$ of $\Sigma$, where each $\partial S_i$ is a circle and none of the $S_i$ is a disc. Then the kernel of \eqref{eq:extend-by-id} is the central subgroup of $\Map(\Sigma)$ freely generated by those Dehn twists $T_{C_i}$ for which $S_i$ is a once-punctured disc.
\end{prop}
\begin{proof}
If $S$ were instead a \emph{finite}-type surface then this would be a special case of \cite[Theorem~3.18]{FaMa11}, which is proven using the Alexander method. (Note that we assume stronger hypotheses than \cite[Theorem~3.18]{FaMa11}, in that we require each component $S_i$ of the closure of the complement $\overline{S \smallsetminus \Sigma}$ to have a single boundary-component.) In our setting, exactly the same proof goes through, using the fact that the Alexander method is valid also for infinite-type surfaces, as proven in \cite{HMV19}.
\end{proof}

\begin{proof}[Proof of Lemma \ref{lem:colimits}]
The statement that $\PMap_\fF(S) \to \Map_\fF(S)$ is injective is obvious, since $\PMap(\Sigma) \subset \Map(\Sigma)$ is injective and we are taking a colimit over the same poset $\fF(S)$ on each side.

To prove that $\Map_\fF(S) \to \Map(S)$ is injective it will suffice, by general properties of colimits, to show that there is a cofinal family of $\Sigma \in \fF(S)$ such that $\Map(\Sigma) \to \Map(S)$ is injective. Let $\Sigma \subset S$ be any properly-embedded finite-type subsurface with $b$ boundary-components $C_1,\ldots,C_b$. We may enlarge it if necessary to ensure that each $C_i$ is a separating curve of $S$. Denote the connected components of (the closure of) $S \smallsetminus \Sigma$ by $S_1,\ldots,S_b$. We now enlarge $\Sigma$ further by taking its union with those $S_i$ that are of finite type (if any). Finally, we may enlarge $\Sigma$ if necessary to ensure that it is not an annulus, by increasing its genus (if $S$ has positive genus) or increasing $b$ (if $S$ has genus zero, in which case it must have infinitely many ends). We are now in the setting of Proposition \ref{prop:injectivity}, which implies that $\Map(\Sigma) \to \Map(S)$ is injective since none of the $S_i$ is a once-punctured disc (indeed, we have ensured that none of the $S_i$ is of finite type). By construction, the $\Sigma \in \fF(S)$ for which we have proven this form a cofinal family in $\fF(S)$, so the result follows.

It is clear by construction that we have $\image(\Map_\fF(S) \to \Map(S)) = \Map_f(S)$ and that
\begin{equation}
\label{eq:clear-inclusion}
\image(\PMap_\fF(S) \to \Map(S)) \supseteq \Map_c(S)
\end{equation}
since compact surfaces are of finite type. What is slightly less clear is the converse of the inclusion \eqref{eq:clear-inclusion}. To see this, suppose that $\varphi \in \Homeo(S)$ represents an element in the image of $\PMap_\fF(S) \to \Map(S)$, so we may assume that it has support contained in some finite-type subsurface $\Sigma \subset S$ and the punctures of $\Sigma$ are fixed pointwise by $\varphi$. Denote by $\Sigma' \subset \Sigma$ a compact subsurface obtained by removing a small open annular neighbourhood of each puncture of $\Sigma$. Since $\varphi$ fixes the punctures of $\Sigma$ pointwise, we may modify it by an isotopy to have support contained in $\Sigma'$, and hence $[\varphi] \in \Map_c(S)$.
This completes the proof of the first point of the lemma.

By general properties of colimits, in order to prove that $\Map_\fC(S) \to \PMap_\fF(S)$ is surjective it suffices to prove that, for any $\Sigma \in \fF(S)$, there exists $\Sigma' \in \fC(S)$ with $\Sigma' \subseteq \Sigma$ such that $\Map(\Sigma') \to \PMap(\Sigma)$ is surjective. The argument in the previous paragraph proves exactly this.

To complete the proof of the second point of the lemma, it now just remains to identify the kernel of $\Map_\fC(S) \to \PMap_\fF(S)$. Since we already know that $\PMap_\fF(S) \to \Map(S)$ is injective, this is the same as the kernel of $\Map_\fC(S) \to \Map(S)$. To identify this, we use Proposition \ref{prop:injectivity} again. Let $\Sigma \subset S$ be any compact subsurface with $b$ boundary-components $C_1,\ldots,C_b$. As before, we may enlarge it if necessary to ensure that each $C_i$ is a separating curve of $S$ and denote the connected components of (the closure of) $S \smallsetminus \Sigma$ by $S_1,\ldots,S_b$. We may enlarge $\Sigma$ by taking its union with those $S_i$ that are compact (if any), and ensure that $\Sigma$ is not an annulus (as before). After doing this, none of the $S_i$ are discs (since we have arranged that none of them are compact) so we are in the setting of Proposition \ref{prop:injectivity}, which tells us that the kernel of $\Map(\Sigma) \to \Map(S)$ is the central subgroup freely generated by those Dehn twists $T_{C_i}$ for which $S_i$ is a once-punctured disc. Taking colimits, it follows that the kernel of $\Map_\fC(S) \to \Map(S)$ is the central subgroup freely generated by all colimits of Dehn twists of this form that arise as we allow the compact subsurface $\Sigma \subset S$ to vary. There is exactly one such colimit of Dehn twists in $\Map_\fC(S)$ for each puncture $p$ of $S$, represented by the family of Dehn twists around $C_\epsilon$ for $\epsilon > 0$, where $C_\epsilon$ is the boundary component surrounding $p$ of a compact subsurface $\Sigma \subset S$ that is locally given by removing an open annulus of radius $\epsilon$ from around $p$. Thus the kernel is a central subgroup with a basis in one-to-one correspondence with the punctures of $S$.
\end{proof}

\begin{rem}
The group $\Map_\fC(S)$ may be a little counterintuitive since it is not a subgroup of a mapping class group in general. As an illustration, we note that it makes sense to consider it also when $S$ is finite-type, for example the once-punctured disc $S = \{ x \in \bR^2 \mid 0 < \lvert x \rvert \leq 1 \}$. The family of annuli $A_\epsilon = \{ x \in \bR^2 \mid \epsilon \leq \lvert x \rvert \leq 1 \}$ for $\epsilon \in (0,1)$ is cofinal in $\fC(S)$, each $\Map(A_\epsilon)$ is infinite cyclic and the homomorphisms $\Map(A_\epsilon) \to \Map(A_{\epsilon'})$ for $\epsilon \geq \epsilon'$ are isomorphisms, so it follows that the colimit $\Map_\fC(S)$ is also infinite cyclic, although $\Map(S)$ is trivial. In this case a generator of $\Map_\fC(S)$ is represented (for example) by the formal colimit of the Dehn twists around the inner boundary components of the annuli $A_\epsilon$, just like at the end of the proof above. More generally, one may see by the same reasoning that $\Map_\fC(S)$, for any finite-type surface $S$, is naturally isomorphic to the mapping class group of the compact surface obtained from $S$ by blowing up each puncture to a boundary component.
\end{rem}

\begin{rem}
As a complement to Lemma \ref{lem:colimits} we discuss briefly the difference between $\Map_c(S)$ and $\Map_f(S)$. If $\varphi$ is a self-homeomorphism of $S$, its induced action on $\Ends(S)$ sends the subset $\cP(S)$ of punctures (cf.~Definition \ref{defn:p} for this notation) onto itself. If $\varphi$ has support contained in a finite-type subsurface, the induced permutation of $\cP(S)$ lies in the subgroup $\Bij_f(\cP(S)) \subseteq \Bij(\cP(S))$ of bijections with finite support. If the induced permutation is trivial, we may shrink the support of $\varphi$ outside of an open neighbourhood of the punctures of $S$, which is then compact, so in this case $[\varphi]$ lies in $\Map_c(S)$. Putting this together, we have a short exact sequence
\begin{equation}
\label{eq:cf-ses}
1 \to \Map_c(S) \longrightarrow \Map_f(S) \longrightarrow \Bij_f(\cP(S)) \to 1.
\end{equation}
Alternatively, this may be deduced as a corollary of Lemma \ref{lem:colimits}. For each properly-embedded finite-type $\Sigma \subset S$ we have a short exact sequence $1 \to \PMap(\Sigma) \to \Map(\Sigma) \to \Bij(\cP(\Sigma)) \to 1$; taking the colimit over $\Sigma \in \fF(S)$ and applying the first part of Lemma \ref{lem:colimits}, we obtain \eqref{eq:cf-ses}.
\end{rem}

We have the following observation, part of which was already stated in Lemma \ref{lem:reformulation}.

\begin{cor}
\label{coro:questions-coincide}
We have the following coincidences of questions.
    \begin{itemize}[nosep]
        \item If $p_S \in \{ 0,1 \}$ then $\Map_c(S) = \Map_f(S)$ and so Questions \ref{q-finite-type-pure} and \ref{q-finite-type} coincide.
        \item If $p_S = 0$ then $\Map_\fC(S) = \PMap_\fF(S)$ and so Questions \ref{q-compact}, \ref{q-finite-type-pure} and \ref{q-finite-type} all coincide.
    \end{itemize}
\end{cor}
\begin{proof}
The first statement follows from the short exact sequence \eqref{eq:cf-ses} and the second statement follows from the second point of Lemma \ref{lem:colimits} (and the first statement).
\end{proof}

\subsection*{Organisation of the proofs.}

We finish this section by briefly describing the overall organisation of the proofs of Theorems \ref{mainthm-LochNess}--\ref{mainthm-genus-zero-infinite-punctures-2}, which occupy \S\ref{s:shiftable}--\S\ref{s:BT-construction}. All of the \emph{vanishing} results are proven in \S\ref{s:shiftable}--\S\ref{s:Cantor-tree-punctured} and all of the \emph{non-vanishing} results are proven in \S\ref{s:non-trivial-classes}--\S\ref{s:BT-construction}, organised as follows:

\begin{itemize}[nosep]
\item[\S\ref{s:shiftable}] --- Theorem \ref{mainthm-LochNess}
\item[\S\ref{s:transferring}] --- Theorem \ref{mainthm-infinite-genus}, except for part \ref{mainthm-infinite-genus}\ref{mainthm-infinite-genus-3}
\item[\S\ref{s:genus-zero}] --- Theorem \ref{mainthm-genus-zero-infinite-punctures-2} -- and hence in particular Corollary \ref{maincor-flute-surface}
\item[\S\ref{s:Cantor-tree-punctured}] --- Theorem \ref{mainthm-genus-zero-finite-punctures}\ref{mainthm-genus-zero-finite-punctures-1}
\item[\S\ref{s:non-trivial-classes}] --- Most of our \emph{non-vanishing} results, namely:
  \begin{itemize}[nosep]
  \item[\S\ref{ss:finite-genus}] --- Theorem \ref{mainthm-finite-positive-genus}
  \item[\S\ref{ss:finite-p-part-1}] --- Theorem \ref{mainthm-genus-zero-finite-punctures}\ref{mainthm-genus-zero-finite-punctures-2}
  \item[\S\ref{ss:topologically-distinguished-set-of-ends}] --- Theorem \ref{mainthm-genus-zero-finite-punctures}\ref{mainthm-genus-zero-finite-punctures-3} and -- more generally -- Theorem \ref{mainthm-genus-zero-infinite-punctures-1}
  \end{itemize}
\item[\S\ref{s:BT-construction}] --- Our last non-vanishing result -- Theorem \ref{mainthm-infinite-genus}\ref{mainthm-infinite-genus-3} -- whose proof has a different flavour from \S\ref{s:non-trivial-classes}.
\end{itemize}

\section{Grid surfaces and shiftable subsurfaces}
\label{s:shiftable}

Most of our vanishing results, including Theorem \ref{mainthm-LochNess}, use the idea of \emph{grid surfaces}. In this section, we introduce this notion, prove the key Proposition \ref{prop:grid-surface} and use it to prove Theorem \ref{mainthm-LochNess}.

\begin{rem}
The proof of Proposition \ref{prop:grid-surface} uses an \emph{infinite iteration argument} that goes back to \cite{Mather1971}, who applied it to the group $\Homeo_c(\bR^d)$ of compactly-supported homeomorphisms of Euclidean space. The argument was axiomatised by \cite{BaumslagDyerHeller1980} into the concept of \emph{mitotic} groups, which are always acyclic. These are related to the concept of the \emph{suspension} of a group, and the argument is therefore sometimes called a \emph{suspension argument}. The argument was further generalised in \cite{Berrick1989} to \emph{binate groups} (which include all mitotic groups), which were also discovered independently (under the name \emph{pseudo-mitotic groups}) by \cite{Varadarajan1985}. A particular class of binate groups is the class of \emph{dissipated groups} \cite{Berrick2002}. See \cite[Section 3]{FLM23} for further information.

In each of those cases, the argument aims to prove the vanishing of the homology of a group, whereas, in our case, we aim to prove that a group homomorphism induces the zero map on homology. This is a little more subtle and requires a kind of ``two-dimensional'' infinite iteration, which we formalise in the notion of \emph{grid surfaces} (Definition \ref{def:grid-surface}). Another effect of this subtlety is that we can only prove our vanishing results on homology with coefficients in a \emph{field}; see Remark \ref{rmk:field-vs-Z} for why this is the case. We note that one could also use \cite[Proposition 1.4]{Varadarajan1985} to prove Proposition \ref{prop:grid-surface}; see Remark \ref{rem-pseudo-mitosis-pf}.
\end{rem}

\begin{defn}
\label{def:grid-surface}
Let $\Sigma$ be a surface with one boundary component. The associated \emph{grid surface} $\Gr(\Sigma)$ is constructed as follows:
\begin{itemize}
    \item Glue an annulus to $\partial \Sigma$ and denote the resulting surface by $\bar{\Sigma}$. Identify $\partial \bar{\Sigma}$ with $\partial [0,1]^2 \subset \bR^2$.
    \item Define $\Gr(\Sigma)$ to be the quotient of $\bZ \times \bN \times \bar{\Sigma}$ that glues the boundaries $\bZ \times \bN \times \partial \bar{\Sigma}$ together in a half-plane grid. See Figure \ref{fig:grid-surface}.
    \item Similarly, define $\Gr_\bZ(\Sigma)$ to be the quotient of $\bZ \times \bZ \times \bar{\Sigma}$ that glues the boundaries $\bZ \times \bZ \times \partial \bar{\Sigma}$ together in a full-plane grid.
\end{itemize}
\end{defn}

\begin{notation}
\label{notation:grid-surfaces}
In the above setting, for a subset $A \subseteq \bN$, we write $\Gr_A(\Sigma)$ for the subsurface of $\Gr(\Sigma)$ given by the image of $\bZ \times A \times \bar{\Sigma}$. For example, see Figure \ref{fig:grid-surface} for illustrations of $\Gr_{[n,\infty)}(\Sigma) =: \Gr_{\geq n}(\Sigma)$, $\Gr_{\{n\}}(\Sigma) =: \Gr_n(\Sigma)$ and $\Gr_{[0,n]}(\Sigma)$.

We also write $\Sigma_{i,j}$ for the $(i,j)$th copy of $\Sigma$ in $\Gr(\Sigma)$. Unless otherwise specified, we shall always identify $\Sigma$ with $\Sigma_{0,0} \subset \Gr(\Sigma)$.
\end{notation}

\begin{rem}
The meaning of the notation $\Sigma_{i,j}$ explained in Notation \ref{notation:grid-surfaces} is used \emph{only} in the present section, and so it should not cause confusion with the more standard meaning of $\Sigma_{g,b}$ to denote the connected, compact, orientable surface of genus $g$ with $b$ boundary-components, which is its meaning in the other sections of this paper.
\end{rem}

\begin{figure}[t]
    \centering
    \includegraphics[scale=0.7]{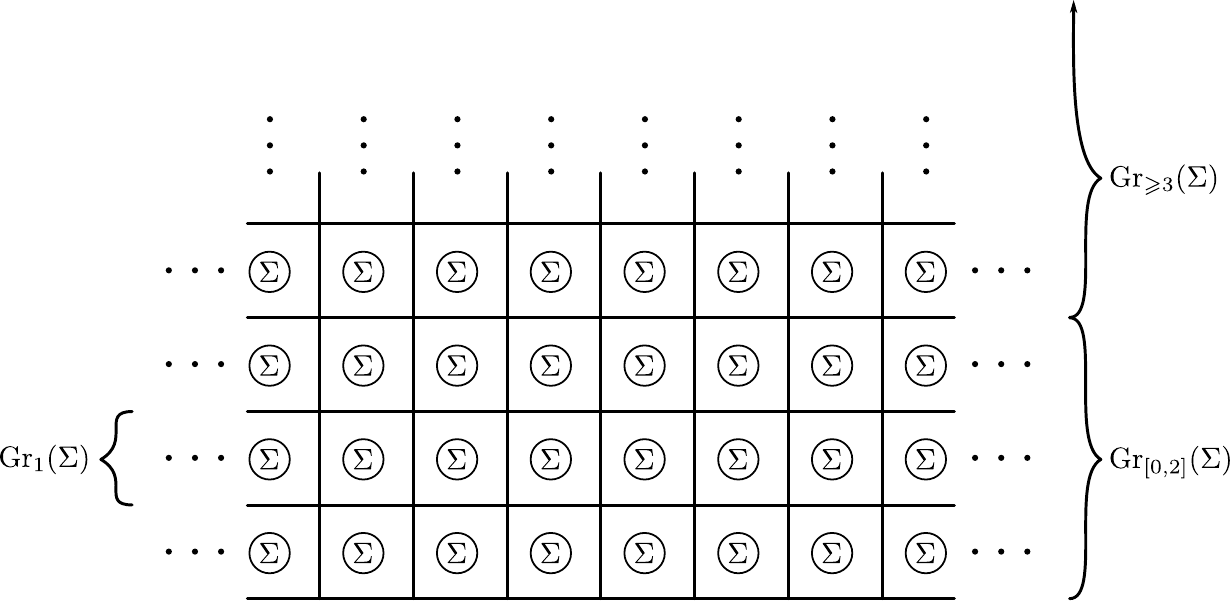}
    \caption{The grid surface $\Gr(\Sigma)$ together with subsurfaces $\Gr_A(\Sigma)$ for various subsets $A \subseteq \bN$. Each circle contains a copy of $\Sigma$; the region between each circle and the boundary of the corresponding square is the annulus in the first point in Definition \ref{def:grid-surface}.}
    \label{fig:grid-surface}
    \vspace{2.5ex}
    \includegraphics[scale=0.7]{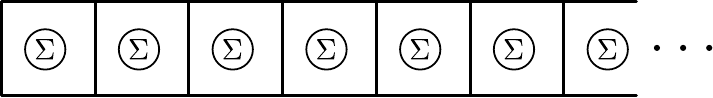}
    \caption{The infinite strip surface $\Str(\Sigma)$ (Definition \ref{def:infinite-strip-surface}).}
    \label{fig:infinite-strip-surface}
\end{figure}

\begin{rem}
The mapping class group of the surface (with non-compact boundary) $\Gr(\Sigma)$ is defined in the usual way, as the group of isotopy classes of homeomorphisms that preserve the boundary pointwise.
\end{rem}

The key technical result of this section is the following.

\begin{prop}
\label{prop:grid-surface}
For any surface $\Sigma$ with one boundary component, the map
\begin{equation}
\label{eq:iota-zero}
\Map(\Sigma) \longrightarrow \Map(\Gr(\Sigma)),
\end{equation}
given by extending homeomorphisms by the identity, induces the zero map on homology with field coefficients in all positive degrees. Hence the same is true also for $\Map(\Sigma) \to \Map(\Gr_\bZ(\Sigma))$.
\end{prop}

\begin{rem}\label{rem-pseudo-mitosis-pf}
We give a direct proof of this proposition below. One could also prove it using the notion of \emph{pseudo-mitosis} \cite[Definition 1.2]{Varadarajan1985}, as follows. One first notices that the embedding $\Map(\Sigma) \hookrightarrow \Map(\Gr_0(\Sigma))$ has a pseudo-mitosis. Hence, by \cite[Proposition 1.4]{Varadarajan1985}, it induces the zero map on homology in degree $1$ for any field coefficients. To promote this to any degree we use the self-similarity of the grid surface. Let $\Str(\Sigma)$ be the strip surface in Figure \ref{fig:infinite-strip-surface}, embedded as a vertical strip in $\Gr(\Sigma)$. Let us take $H = \Map(\Str(\Sigma))$ and $G = \Map(\Gr(\Sigma))$ in \cite[Proposition 1.4]{Varadarajan1985}; the ``horizontal translation'' (modified in a neighbourhood of the boundary line so as to fix it pointwise) is then part of a pseudo-mitosis for $H \subset G$. Taking $A \subset H$ to be $\Map(\Sigma) \hookrightarrow \Map(\Str(\Sigma))$, we then notice that this embedding induces the zero map on homology in degree $1$ for any field coefficients by what we showed above, since we may factor it through the embedding $\Map(\Sigma) \hookrightarrow \Map(\Gr_0(\Sigma))$ using the inclusion $\Gr_0(\Sigma) \subset \Gr(\Sigma)$ and a homeomorphism $\Str(\Sigma) \cong \Gr(\Sigma)$ that is the identity on the preferred embedded copy of $\Sigma$ (see for example Lemma \ref{lem:strip-and-grid} below). Hence \cite[Proposition 1.4]{Varadarajan1985} implies that the composition $A \subset H \subset G$, which is the embedding $\Map(\Sigma) \hookrightarrow \Map(\Gr({\Sigma}))$, induces the zero map on homology in degrees $1$ and $2$ with any field coefficients. Iterating this trick, we conclude inductively that $\Map(\Sigma) \hookrightarrow \Map(\Gr({\Sigma}))$ induces the zero map on homology in all positive degrees with any field coefficients.
\end{rem}

In order to apply Proposition \ref{prop:grid-surface} in examples, it will be useful to have a simpler description of $\Gr(\Sigma)$.

\begin{defn}
\label{def:infinite-strip-surface}
Let $\Sigma$ be a surface with one boundary component. The \emph{infinite strip surface} $\Str(\Sigma)$ is constructed, similarly to Definition \ref{def:grid-surface}, to be the quotient of $\bN \times \bar{\Sigma}$ that glues the boundaries $\bN \times \partial\bar{\Sigma}$ together in a one-dimensional ray. See Figure \ref{fig:infinite-strip-surface}.
\end{defn}

Clearly $\Str(\Sigma)$ embeds properly into $\Gr(\Sigma)$ (compare Figures \ref{fig:grid-surface} and \ref{fig:infinite-strip-surface}). But in fact we have:

\begin{lem}
\label{lem:strip-and-grid}
The surfaces $\Str(\Sigma)$ and $\Gr(\Sigma)$ are homeomorphic. Moreover, this homeomorphism may be chosen to act by the identity on the preferred embedded copy of $\Sigma$, namely the left-most copy for $\Str(\Sigma)$ and the copy at coordinates $(0,0)$ for $\Gr(\Sigma)$.
\end{lem}
\begin{proof}
Let $S \subset \Str(\Sigma)$ be the complement of a closed collar neighbourhood (this is, of course, homeomorphic to $\Str(\Sigma)$). It will suffice to describe a proper embedding of $S$ into $\Gr(\Sigma)$ such that the complement of its image is a closed collar neighbourhood of $\Gr(\Sigma)$. Such a proper embedding may be constructed easily as soon as one chooses a bijection $\upsilon \colon \bN \to \bZ \times \bN$ such that $\upsilon(n)$ and $\upsilon(n+1)$ are neighbours (at $\ell^1$-distance $1$ from each other) for every $n$. To ensure that the second statement of the lemma holds, we must also arrange that $v(0) = (0,0)$. For example one may take the ``snake bijection'' that progressively fills each $\ell^\infty$-ball around $(0,0)$. Alternatively, the fact that $\Str(\Sigma)$ and $\Gr(\Sigma)$ are homeomorphic may be deduced from the classification of surfaces with non-compact boundary \cite{BrownMesser1979} (although quoting this much more general classification result is overkill here).
\end{proof}

\begin{defn}
\label{def:shiftable}
A properly-embedded subsurface $\Sigma \subset S$ is called \emph{shiftable} if the inclusion $\Sigma \subset S$ extends to a proper embedding $\Str(\Sigma) \hookrightarrow S$.
\end{defn}

\begin{rem}
Elsewhere, a subsurface $\Sigma \subset S$ is sometimes called ``shiftable'' if there is a homeomorphism of $S$ such that all of the iterated images of $\Sigma$ under this homeomorphism are pairwise disjoint. In fact, these two definitions are equivalent, although we will not need this equivalence here. In one direction, suppose that $\Sigma \subset S$ is shiftable in the sense of Definition \ref{def:shiftable}. By Lemma \ref{lem:strip-and-grid}, we therefore have a proper embedding $\Gr_0(\Sigma) \subset \Gr(\Sigma) \cong \Str(\Sigma) \hookrightarrow S$. The evident shift homeomorphism of $\Gr_0(\Sigma)$ may then be extended by the identity to obtain a shift homeomorphism of $S$ for $\Sigma$. In the other direction:
\end{rem}

\begin{lem}
Suppose that $\Sigma \subset S$ is a properly-embedded subsurface with non-empty, connected boundary that admits a shift homeomorphism, i.e.\ a homeomorphism $f \colon S \to S$ such that $f^n(\Sigma) \cap \Sigma = \varnothing$ for all $n\geq 1$. Then $\Sigma \subset S$ is shiftable in the sense of Definition \ref{def:shiftable}.
\end{lem}
\begin{proof}
Let us denote by $T \subset S$ the surface obtained from $S$ by removing the interior of each $f^n(\Sigma)$ for $n\geq 0$ and write $B_n = f^n(\partial \Sigma)$, so that the boundary of $T$ is the disjoint union of the $B_n$ for $n\geq 0$. Also write $U$ for the surface (without boundary) obtained from $T$ by collapsing each $B_n$ to a point $b_n \in U$.

In order to show that $\Sigma \subset S$ extends to a proper embedding $\Str(\Sigma) \hookrightarrow S$, it will suffice to find a locally finite, pairwise disjoint collection of arcs $\alpha_i$ in $T$, for $i\geq 0$, such that $\alpha_i$ connects $B_{n_i}$ to $B_{n_{i+1}}$ for some increasing sequence $n_i$. This is because, given such a collection of arcs, the union of the $f^{n_i}(\Sigma)$ and tubular neighbourhoods of the arcs $\alpha_i$ for $i\geq 0$ will give the desired proper embedding $\Str(\Sigma) \hookrightarrow S$. Passing to the surface $U$ and the sequence of points $(b_n)_{n\geq 0}$ in $U$, it is therefore sufficient to find a locally finite collection of arcs $\alpha_i$ in $U$, connecting $b_{n_i}$ to $b_{n_{i+1}}$ for some subsequence $(b_{n_i})_{i\geq 0}$ of $(b_n)_{n\geq 0}$, that are pairwise disjoint except at their endpoints.

To do this, choose any subsequence that converges in the Freudenthal compactification $\overline{U}$ of $U$ to an end $e \in \overline{U} \smallsetminus U = \Ends(U)$. Write $U$ as an increasing union of compact, connected subsurfaces $C_k \subset U$ for $k\geq 0$ and let $U_k \subset U$ be the connected component of $U \smallsetminus C_k$ such that $e$ is a limit point of $U_k$ in $\overline{U}$. We may choose the $C_k$ such that each $U_k$ has a single boundary component $A_k$, which implies that each stratum $U_k \smallsetminus \mathrm{int}(U_{k+1})$ is connected and has two boundary components $A_k$ and $A_{k+1}$. Since $b_{n_i} \to e$ in $\overline{U}$, we may pass to appropriate subsequences of $C_k$ and $b_{n_i}$ to arrange that $b_{n_i} \in U_i \smallsetminus U_{i+1}$ for each $i\geq 0$. Choose a point $x_i$ on each circle $A_i$ and choose arcs $\beta_i$ and $\gamma_i$ in $U_i \smallsetminus \mathrm{int}(U_{i+1})$ such that $\beta_i(0) = x_i$, $\beta_i(1) = \gamma_i(0) = b_{n_i}$, $\gamma_i(1) = x_{i+1}$ and $\beta_i,\gamma_i$ are disjoint except at the point $b_{n_i}$. The desired collection of arcs $\alpha_i$ is then obtained by gluing $\gamma_i$ to $\beta_{i+1}$ for each $i\geq 0$.
\end{proof}

\begin{cor}
\label{coro:shiftable}
Let $\Sigma \subset S$ be a properly-embedded subsurface and suppose that it is shiftable. Then the natural map $\Map(\Sigma) \to \Map(S)$ induces the zero map on homology with field coefficients in all positive degrees.
\end{cor}
\begin{proof}
Since $\Sigma$ is shiftable, the map $\Map(\Sigma) \to \Map(S)$ factors as $\Map(\Sigma) \to \Map(\Str(\Sigma)) \to \Map(S)$. Hence the result follows from Proposition \ref{prop:grid-surface} and Lemma \ref{lem:strip-and-grid}.
\end{proof}

\begin{proof}[Proof of Proposition \ref{prop:grid-surface}]
The second statement of the proposition follows from the first statement, since $\Map(\Sigma) \to \Map(\Gr_\bZ(\Sigma))$ factors through \eqref{eq:iota-zero}.

To prove the first statement, we first define various homomorphisms that we shall need. For $m \in \bZ$ and $n \in \bN$, let
\begin{equation}
\label{eq:psi-bar}
\bar{\psi}_{m,n} \colon \Map(\Sigma) \longrightarrow \Map(\Gr_n(\Sigma))
\end{equation}
be the homomorphism that sends $[\varphi]\in \Map(\Sigma)$ to the mapping class represented by the homeomorphism of $\Gr_n(\Sigma)$ that acts by $\varphi$ on $\Sigma_{i,n}$ for each $i\geq m$ and by the identity elsewhere. We also write
\begin{equation}
\label{eq:psi}
\psi_{m,n} \colon \Map(\Sigma) \longrightarrow \Map(\Gr_{\geq n}(\Sigma))
\end{equation}
for the composition of $\bar{\psi}_{m,n}$ with the natural homomorphism $\Map(\Gr_n(\Sigma)) \to \Map(\Gr_{\geq n}(\Sigma))$ given by extension by the identity. We write
\begin{equation}
\label{eq:iota}
\iota_n \colon \Map(\Sigma) \longrightarrow \Map(\Gr_{\geq n}(\Sigma))
\end{equation}
for the homomorphism sending $[\varphi]$ to the mapping class represented by the homeomorphism of $\Gr_{\geq n}(\Sigma)$ that acts by $\varphi$ on $\Sigma_{0,n}$ and by the identity elsewhere. Note that $\iota_0$ is precisely the map \eqref{eq:iota-zero} in Proposition \ref{prop:grid-surface}. Finally, we define
\begin{equation}
\label{eq:eta-and-nu}
\eta_n \text{ and } \nu_n \colon \Map(\Sigma) \times \Map(\Sigma) \longrightarrow \Map(\Gr_{\geq n} (\Sigma))
\end{equation}
to send $([\varphi_1],[\varphi_2])$ to the mapping class represented by the homeomorphism of $\Gr_{\geq n}(\Sigma)$ that acts:
\begin{itemize}[nosep]
\item by $\varphi_2$ on $\Sigma_{i,n}$ for all $i\geq 1$,
\item (for $\eta_n$:) by $\varphi_1$ on $\Sigma_{0,n}$,
\item (for $\nu_n$:) by $\varphi_1$ on $\Sigma_{0,n+1}$,
\item by the identity elsewhere.
\end{itemize}
The proof will use the following commutative diagram.
\begin{equation}
\label{eq:big-diagram}
\begin{split}
\centering
\begin{tikzpicture}
[x=1mm,y=1.5mm]
\node (l) at (0,10) {$\Map(\Sigma)$};
\node (mm) at (45,10) {$\Map(\Sigma) \times \Map(\Sigma)$};
\node (mb) at (45,0) {$\Map(\Gr_{\geq n+1}(\Sigma)) \times \Map(\Gr_n(\Sigma))$};
\node (rt) at (105,20) {$\Map(\Gr_{\geq n}(\Sigma))$};
\node (rm) at (105,10) {$\Map(\Gr_{\geq n}(\Sigma))$};
\node (rb) at (105,0) {$\Map(\Gr_{\geq n}(\Sigma))$};
\draw[->] (l) to node[above,font=\small]{$\Delta$} (mm);
\draw[->] (mm) to node[above,font=\small]{$\eta_n$} (rm);
\draw[->] (mb) to node[below,font=\small]{glue} (rb);
\draw[->] (mm) to node[left,font=\small]{$\iota_{n+1} \times \bar{\psi}_{1,n}$} (mb);
\draw[->] (mm) to node[above,font=\small]{$\nu_n$} (rb);
\draw[<->] (rm) to node[right,font=\small]{$c_{\mathrm{rot}}$} (rb);
\draw[<->] (rm) to node[right,font=\small]{$c_{\mathrm{sh}}$} (rt);
\draw[->] (l) -- (0,20) -- node[above,font=\small]{$\psi_{1,n}$} (rt);
\end{tikzpicture}
\end{split}
\end{equation}
Here, $\Delta$ denotes the diagonal map and ``glue'' is the map that takes two homeomorphisms defined on $\Gr_{\geq n+1}(\Sigma)$ and on $\Gr_n(\Sigma)$ and glues them to a homeomorphism on $\Gr_{\geq n}(\Sigma) = \Gr_{\geq n+1}(\Sigma) \cup \Gr_n(\Sigma)$. The right-hand vertical maps $c_{\mathrm{sh}}$ and $c_{\mathrm{rot}}$ are conjugation by the (vertically bounded) homeomorphisms $\mathrm{sh} \text{ and } \mathrm{rot} \colon \Gr_{\geq n}(\Sigma) \to \Gr_{\geq n}(\Sigma)$ defined, respectively, by \emph{shifting one step to the right on the $n$th row} and by \emph{rotating by ninety degrees in the subsurface containing $\Sigma_{i,j}$ for $i \in \{-1,0\}$ and $j \in \{n,n+1\}$}. (These maps in diagram \eqref{eq:big-diagram} are depicted with double arrows; the direction of the arrow depends on which side one conjugates on and (in the case of $c_{\mathrm{rot}}$) the sense of rotation of the subsurface.)

The statement that we shall prove -- by induction on $j$ -- is the following. Let us fix a field $K$. Then for every $n \in \bN$ and $j\geq 1$, the induced map
\begin{equation}
\label{eq:inductive-hypothesis}
(\iota_n)_* \colon H_j(\Map(\Sigma);K) \longrightarrow H_j(\Map(\Gr_{\geq n}(\Sigma));K)
\end{equation}
is the zero map. In particular, this will complete the proof of the proposition, which corresponds to the special case of $n=0$. The base case $j=0$ is vacuous, so we let $j\geq 1$, fix any $n \in \bN$ and assume as inductive hypothesis that \eqref{eq:inductive-hypothesis} is the zero map for smaller values of $j$ and for all values of $n$.

Let us apply the Künneth theorem to the product of maps $\iota_{n+1} \times \bar{\psi}_{1,n}$ in diagram \eqref{eq:big-diagram}. It implies that we have a commutative square
\begin{center}
\small
\begin{tikzpicture}
[x=1mm,y=1.5mm]
\node (lt) at (0,10) {$\displaystyle\smash{\bigoplus_{k=0}^j}\, H_k(\Map(\Sigma)) \otimes H_{j-k}(\Map(\Sigma))$};
\node (rt) at (90,10) {$\displaystyle\smash{\bigoplus_{k=0}^j}\, H_k(\Map(\Gr_{\geq n+1}(\Sigma))) \otimes H_{j-k}(\Map(\Gr_n(\Sigma)))$};
\node (lm) at (0,0) {$H_j(\Map(\Sigma) \times \Map(\Sigma))$};
\node (rm) at (90,0) {$H_j(\Map(\Gr_{\geq n+1}(\Sigma)) \times \Map(\Gr_n(\Sigma)))$};
\draw[->] (lt) to (rt);
\node at ($ (lt.east)!0.5!(rt.west) + (0,3) $) [font=\footnotesize] {$\bigoplus (\iota_{n+1})_* \otimes (\bar{\psi}_{1,n})_*$};
\draw[->] (lm) to node[above,font=\footnotesize]{$(\iota_{n+1} \times \bar{\psi}_{1,n})_*$} (rm);
\draw[->] (lt) to node[left,font=\footnotesize]{\rotatebox{90}{$\cong$}} (lm);
\draw[->] (rt) to node[right,font=\footnotesize]{\rotatebox{270}{$\cong$}} (rm);
\end{tikzpicture}
\end{center}
in which the vertical maps are isomorphisms and the coefficients of homology are $K$ in each case. Let $\alpha \in H_j(\Map(\Sigma);K)$ be any element. Naturality of the K\"unneth decomposition, applied to the two projections $\Map(\Sigma) \times \Map(\Sigma) \twoheadrightarrow \Map(\Sigma)$, implies that the image of $\Delta_*(\alpha)$ in the top-left corner of this square has $0$-th component equal to $1 \otimes \alpha$ and $j$-th component equal to $\alpha \otimes 1$. The inductive hypothesis implies that the top horizontal map is the zero map on the $k$-th component for all $0<k<j$. It follows that the image of $\Delta_*(\alpha)$ in the top-right corner of the square is $1 \otimes (\bar{\psi}_{1,n})_*(\alpha) + (\iota_{n+1})_*(\alpha) \otimes 1$. Composing this with the right-hand vertical isomorphism and the map on homology induced by the ``glue'' map of \eqref{eq:big-diagram}, we obtain the element $(\psi_{1,n})_*(\alpha) + (\iota_{n+1})_*(\alpha) \in H_j(\Map(\Gr_{\geq n}(\Sigma));K)$. It therefore follows that the map on $H_j(-;K)$ induced by the map across diagram \eqref{eq:big-diagram} is equal to $(\psi_{1,n})_* + (\iota_{n+1})_*$. But it is also equal to $(\psi_{1,n})_*$, so we must have $(\iota_{n+1})_* = 0$. Since $\iota_n$ and $\iota_{n+1}$ are conjugate as maps $\Map(\Sigma) \to \Map(\Gr_{\geq n}(\Sigma))$, it follows that also $(\iota_n)_* = 0$, as claimed.
\end{proof}

\begin{rem}
\label{rmk:field-vs-Z}
The obstruction to upgrading our vanishing results from field coefficients to arbitrary (in particular, integral) coefficients is due to the failure of \emph{naturality} of the Künneth decomposition (i.e.~the failure of the Künneth short exact sequence to admit a natural splitting), which prevents the last paragraph of the above proof from going through unless one knows that the Tor terms vanish.
\end{rem}

In \S\ref{s:transferring} we will apply Corollary \ref{coro:shiftable} to prove the vanishing results of Theorem \ref{mainthm-infinite-genus}. We finish this section by proving, directly from Proposition \ref{prop:grid-surface}, the special case of Theorem \ref{mainthm-infinite-genus} corresponding to Theorem \ref{mainthm-LochNess}.

\begin{proof}[Proof of Theorem \ref{mainthm-LochNess}]
Let $\Sigma$ be a compact subsurface of $L$, the Loch Ness monster surface. Our goal is to prove that the homomorphism $\Map(\Sigma) \to \Map(L)$ induces the zero map on homology with field coefficients in all positive degrees. (I.e.~a negative answer to Question \ref{q-compact} for $L$, which is equivalent to Question \ref{q-finite-type-pure} for $L$ by Lemma \ref{lem:reformulation} since $L$ has no punctures.) By including $\Sigma$ into a larger compact subsurface if necessary, we may assume that it has exactly one boundary component and positive genus. The pair $(L,\Sigma)$ is homeomorphic to the pair $(\Gr_\bZ(\Sigma),\Sigma)$, so the result follows from Proposition \ref{prop:grid-surface}.
\end{proof}

\section{Transferring homology classes to shiftable subsurfaces}
\label{s:transferring}

In this section we generalise Theorem \ref{mainthm-LochNess} by proving the vanishing results of Theorem \ref{mainthm-infinite-genus} (namely all of Theorem \ref{mainthm-infinite-genus} except for part \ref{mainthm-infinite-genus}\ref{mainthm-infinite-genus-3}, which we prove later in \S\ref{s:BT-construction}). This depends fundamentally on Corollary \ref{coro:shiftable} from the previous section, together with a technique (Proposition \ref{prop:transfer-to-shiftable}) to transfer the support of homology classes to shiftable subsurfaces using Harer's homological stability results for mapping class groups of finite-type surfaces.

\begin{prop}
\label{prop:transfer-to-shiftable}
Suppose that $g_S = \infty$ and let $\Sigma \subset S$ be a properly-embedded finite-type subsurface of $S$. If $\Sigma$ is not compact, then we additionally assume either that $p_S = 0$ or that $S$ has a mixed end. For each integer $i\geq 1$, there exists another properly-embedded subsurface $\Sigma' \subseteq S$ such that:
\begin{enumerate}[nosep,label=\textup{(\arabic*)}]
\item\label{prop:transfer:intersection} $\Sigma \cap \Sigma'$ is an interval in $\partial\Sigma$ and in $\partial\Sigma'$;
\item\label{prop:transfer:shiftable} $\Sigma'$ is shiftable in $S$;
\item\label{prop:transfer:extension-map-surjective} the extension map $\Map(\Sigma') \to \Map(\Sigma \cup \Sigma')$ is surjective on homology up to degree $i$.
\end{enumerate}
\end{prop}

This proposition, along with Corollary \ref{coro:shiftable}, quickly implies the vanishing results of Theorem \ref{mainthm-infinite-genus}.

\begin{proof}[Proof of the vanishing results of Theorem \ref{mainthm-infinite-genus} assuming Proposition \ref{prop:transfer-to-shiftable}]
Let $S$ and its subsurface $\Sigma \subset S$ be as in Proposition \ref{prop:transfer-to-shiftable}; we need to prove that the homomorphism $\Map(\Sigma) \to \Map(S)$ induces the zero map on homology with field coefficients in all positive degrees. Let $K$ be a field and fix a homological degree $i\geq 1$. Let the subsurface $\Sigma' \subset S$ be as in the conclusion of Proposition \ref{prop:transfer-to-shiftable}. Since $\Sigma'$ is shiftable, we know from Corollary \ref{coro:shiftable} that the induced map $H_i(\Map(\Sigma');K) \to H_i(\Map(S);K)$ is zero. Since the intersection of $\Sigma$ and $\Sigma'$ is an interval in each of their boundaries, their union in $S$ is their boundary connected sum, and we may consider the extension map $\Map(\Sigma') \to \Map(\Sigma \cup \Sigma')$, which by part \ref{prop:transfer:extension-map-surjective} of Proposition \ref{prop:transfer-to-shiftable} induces a surjection $H_i(\Map(\Sigma');K) \twoheadrightarrow H_i(\Map(\Sigma\cup\Sigma');K)$. From the commutative diagram of homomorphisms induced by extension maps
\[
\begin{tikzcd}
H_i(\Map(\Sigma);K) \ar[dr] \ar[drrr] &&& \\
& H_i(\Map(\Sigma\cup\Sigma');K) \ar[rr] && H_i(\Map(S);K) \\
H_i(\Map(\Sigma');K) \ar[ur,two heads] \ar[urrr,"0",swap] &&&
\end{tikzcd}
\]
it then follows that $H_i(\Map(\Sigma);K) \to H_i(\Map(S);K)$ is also the zero map.
\end{proof}

The proof of Proposition \ref{prop:transfer-to-shiftable} has two ingredients: Lemma \ref{lem:infinite-genus-shiftable} and Theorem \ref{thm:hmstb}.

\begin{lem}
\label{lem:infinite-genus-shiftable}
Let $S$ and its subsurface $\Sigma \subset S$ be as in Proposition \ref{prop:transfer-to-shiftable}. If $\Sigma$ has exactly one boundary component, then it is shiftable.
\end{lem}
\begin{proof}
For any infinite-type surface $S$, it follows from the construction in \cite[\S 5]{Ri63} that, if $e$ is a \emph{non-planar} end of $S$ and $g$ is any non-negative integer then there exists a proper embedding $\Str(\Sigma_{g,1}) \hookrightarrow S$ of the infinite strip surface such that all unbounded sequences in $\Str(\Sigma_{g,1})$ converge to $e$ in the Freudenthal compactification $\overline{S}$ of $S$ (cf.\ \S\ref{ss:classification}). Similarly, if $e$ is a \emph{mixed} end of $S$ (Definition \ref{def:mixed-end}) and $g,n$ are non-negative integers then there exists a proper embedding $\Str(\Sigma_{g,1}^n) \hookrightarrow S$ such that all unbounded sequences in $\Str(\Sigma_{g,1}^n)$ converge to $e$ in $\overline{S}$.

Putting ourselves now in the setting of Lemma \ref{lem:infinite-genus-shiftable}, suppose first that $\Sigma$ is compact, so it is homeomorphic to $\Sigma_{g,1}$ for some $g\geq 0$. Since $g_S = \infty$ there is at least one non-planar end $e$ of $S$, so we may choose a proper embedding $\Str(\Sigma_{g,1}) \subset S$ as in the previous paragraph. Since the property of being shiftable is preserved under self-homeomorphisms of $S$, we may assume by applying an appropriate self-homeomorphism of $S$ that the subsurface $\Sigma \subset S$ is the subsurface of $\Str(\Sigma_{g,1}) \subset S$ corresponding to the left-most copy of $\Sigma_{g,1}$ in the infinite strip (cf.~Figure \ref{fig:infinite-strip-surface}). Thus $\Sigma \subset S$ is shiftable.

Now suppose that $\Sigma$ is non-compact, so it is homeomorphic to $\Sigma_{g,1}^n$ for some $g\geq 0$ and $n\geq 1$. This implies that $p_S > 0$, which by assumption means that $S$ has a mixed end $e$, and so we may choose a proper embedding $\Str(\Sigma_{g,1}^n) \subset S$ as in the first paragraph of the proof. As above, we may assume by applying a self-homeomorphism of $S$ that the subsurface $\Sigma \subset S$ is the subsurface of $\Str(\Sigma_{g,1}^n) \subset S$ corresponding to the left-most copy of $\Sigma_{g,1}^n$ in the infinite strip. Thus $\Sigma \subset S$ is shiftable.
\end{proof}

The second ingredient is a collection of homological stability results for mapping class groups of connected, finite-type, orientable surfaces. We recall just the statements about \emph{surjectivity}, since these are all that we shall need.

\begin{thm}
\label{thm:hmstb}
The genus-increasing, boundary-component-increasing, puncture-increasing and capping maps, which are each defined by extending homeomorphisms by the identity, induce surjections on homology in the following ranges of degrees.
\begin{enumerate}[nosep,label=\textup{(\arabic*)}]
\item\label{thm:hmstb-genus} The map $H_i(\Sigma_{g,b}^n) \to H_i(\Sigma_{g,b}^n \natural \Sigma_{1,1}) = H_i(\Sigma_{g+1,b}^n)$ is surjective for $g\geq \frac{3}{2}i$.
\item\label{thm:hmstb-bdy} The map $H_i(\Sigma_{g,b}^n) \to H_i(\Sigma_{g,b}^n \natural \Sigma_{0,2}) = H_i(\Sigma_{g,b+1}^n)$ is surjective for $g\geq \frac{3}{2}i$.
\item\label{thm:hmstb-capping} The map $H_i(\Sigma_{g,1}^n) \to H_i(\Sigma_{g}^n)$ filling the boundary circle with a disc is surjective for $g\geq \frac{3}{2}(i-1)$.
\item\label{thm:hmstb-puncture} The map $H_i(\Sigma_{g,b}^n) \to H_i(\Sigma_{g,b}^n \natural \Sigma_{0,1}^1) = H_i(\Sigma_{g,b}^{n+1})$ is surjective for $n\geq 2i$.
\end{enumerate}
\end{thm}
\begin{proof}
Parts \ref{thm:hmstb-genus}--\ref{thm:hmstb-capping} are all due to Harer \cite[Theorem 0.1]{Har85}, except with a larger lower bound on $g$ (Harer does not directly consider the capping map, but part \ref{thm:hmstb-capping} follows indirectly from his results about his map called $\eta$). Improvements to this lower bound were made by Ivanov \cite{Ivanov1993}, Boldsen \cite{Boldsen2012} and Randal-Williams \cite{Randal-Williams2016}; see also the survey by Wahl \cite{Wahl_survey}, which gives the best-known ranges. Part \ref{thm:hmstb-puncture} is due to Hatcher-Wahl \cite[Proposition~1.5]{HW10}.
\end{proof}

\begin{proof}[Proof of Proposition \ref{prop:transfer-to-shiftable}]
Assume first that $\Sigma$ is compact, so it is homeomorphic to $\Sigma_{g,b}$ for some $g\geq 0$ and $b\geq 1$. Since $g_S = \infty$, the definition of $g_S$ implies that we may find another subsurface $\Sigma'' \subset S$, disjoint from $\Sigma$, that is homeomorphic to $\Sigma_{h,1}$ for a genus $h$ as large as we choose. Let us choose $h \geq \frac{3}{2}i$. Since $S$ is path-connected, we may choose a path from a point on the boundary of $\Sigma$ to a point on the boundary of $\Sigma''$ and whose interior is contained in $S \smallsetminus (\Sigma \sqcup \Sigma'')$. Let $\Sigma'$ be the union of $\Sigma''$ and a tubular neighbourhood of this path; this is again homeomorphic to $\Sigma_{h,1}$. It satisfies condition \ref{prop:transfer:intersection} of the proposition by construction. Since it has exactly one boundary component, it satisfies condition \ref{prop:transfer:shiftable} of the proposition by Lemma \ref{lem:infinite-genus-shiftable}. Since $h \geq \frac{3}{2}i$, it satisfies condition \ref{prop:transfer:extension-map-surjective} of the proposition by parts \ref{thm:hmstb-genus} and \ref{thm:hmstb-bdy} of Theorem \ref{thm:hmstb}, since the extension map $\Map(\Sigma') \to \Map(\Sigma\cup\Sigma')$ may be factored into finitely many genus-increasing maps and finitely many boundary-component-increasing maps.

Now suppose that $\Sigma$ is non-compact, so it is homeomorphic to $\Sigma_{g,b}^n$ for some $g\geq 0$ and $b,n\geq 1$. This implies that $S$ has at least one puncture, i.e.~$p_S > 0$, so by assumption $S$ has a mixed end, in particular $p_S = \infty$. The proof is then the same as in the previous paragraph, except that we choose $\Sigma''$ to be homeomorphic to $\Sigma_{h,1}^m$ for $h \geq \frac{3}{2}i$ and $m\geq 2i$, using the fact that $g_S = p_S = \infty$. The rest of the proof is then identical, except that to verify condition \ref{prop:transfer:extension-map-surjective} of the proposition we also need part \ref{thm:hmstb-puncture} of Theorem \ref{thm:hmstb}, factoring the extension map $\Map(\Sigma') \to \Map(\Sigma\cup\Sigma')$ into finitely many genus-increasing maps, boundary-component-increasing maps and puncture-increasing maps.
\end{proof}

\begin{rem}
Part \ref{thm:hmstb-bdy} of Theorem \ref{thm:hmstb} is notable in that, when increasing the number $b$ of boundary components, the range in which homological stability holds depends on the genus $g$, not on $b$. This was crucial in the proof of Proposition \ref{prop:transfer-to-shiftable} above, since we were free to choose $\Sigma'$ to have as high genus and as many punctures as necessary, but it had to have a single boundary component, in order to be able to apply Lemma \ref{lem:infinite-genus-shiftable}.
\end{rem}

\section{Genus zero surfaces with countably infinitely many punctures}
\label{s:genus-zero}

In this section we prove Theorem \ref{mainthm-genus-zero-infinite-punctures-2}, concerning the case when $S$ has genus zero and its space of ends is a closed ordinal interval of the form $[0,\omega^\alpha]$. The proof is different when $\alpha$ is a (countable) successor ordinal and when it is a (countable) limit ordinal; we will deal with these two cases separately -- see Proposition \ref{prop:omega-alpha-successor} and Corollary \ref{cor:omega-alpha-limit}.

\begin{defn}
\label{def:Sigma-alpha}
For a countable ordinal $\alpha$, let us write $\Sigma(\alpha) = \bS^2 \smallsetminus [0,\omega^\alpha]$ and $\Sigma^\circ(\alpha) = \bD^2 \smallsetminus [0,\omega^\alpha]$. In other words, up to homeomorphism, $\Sigma(\alpha)$ is the unique genus-zero surface whose space of ends is homeomorphic to $[0,\omega^\alpha]$ and $\Sigma^\circ(\alpha)$ is the result of removing the interior of a closed disc from $\Sigma(\alpha)$.
\end{defn}

\subsection{Successor ordinals.}

Let us first suppose that $\alpha$ is a successor ordinal, in other words $\alpha = \beta + 1$ for some ordinal $\beta$. In this case, $\Sigma(\alpha)$ may be realised as a (full-plane) grid surface:

\begin{lem}
\label{lem:Sigma-alpha-grid-surface}
There is a homeomorphism $\Gr_\bZ(\Sigma^\circ(\beta)) \cong \Sigma(\alpha)$.
\end{lem}
\begin{proof}
Clearly $\Gr_\bZ(\Sigma^\circ(\beta))$ has genus zero, so by the classification of surfaces it suffices to show that its space of ends is homeomorphic to $[0,\omega^\alpha]$. By construction, its space of ends is the one-point compactification of disjoint union of countably infinitely many copies of $[0,\omega^\beta]$; by Lemma \ref{lem:ordinal-interval-calculations} this is $[0,\omega^\alpha]$.
\end{proof}

Half of Theorem \ref{mainthm-genus-zero-infinite-punctures-2} -- the case when $\alpha$ is a successor ordinal -- is given by the following.

\begin{prop}
\label{prop:omega-alpha-successor}
Suppose that $\alpha$ is a countable successor ordinal. Then the inclusion
\[
\iota_f(\Sigma(\alpha)) \colon \Map_f(\Sigma(\alpha)) \longhookrightarrow \Map(\Sigma(\alpha))
\]
induces the zero map on homology in positive degrees with any field coefficients.
\end{prop}
\begin{proof}
Let $\Sigma \subset \Sigma(\alpha)$ be a properly-embedded subsurface of finite type and denote by $\iota$ the homomorphism $\Map(\Sigma) \to \Map(\Sigma(\alpha))$ given by extending homeomorphisms by the identity. Identifying $\Sigma(\alpha)$ with the grid surface $\Gr_\bZ(\Sigma^\circ(\beta))$ by Lemma \ref{lem:Sigma-alpha-grid-surface}, the subsurface $\Sigma \subset \Gr_\bZ(\Sigma^\circ(\beta))$ must be bounded, since it is of finite type and therefore must be bounded away from the non-isolated end ``at infinity''. Hence $\Sigma$ is contained in a sub-square of the grid of side-length $n$ for some $n \geq 1$. Zooming out by a factor of $n$, we may identify $\Gr_\bZ(\Sigma^\circ(\beta))$ with the grid surface $\Gr_\bZ(\natural^{n^2}\Sigma^\circ(\beta))$, in which each ``piece'' of the grid is the boundary connected sum of $n^2$ copies of $\Sigma^\circ(\beta)$. Applying an appropriate shift homeomorphism, we may assume that $\Sigma$ is contained in the copy of $\natural^{n^2}\Sigma^\circ(\beta)$ at the coordinates $(0,0)$ in the grid. The homomorphism $\iota$ therefore factors as
\[
\Map(\Sigma) \longrightarrow \Map(\natural^{n^2}\Sigma^\circ(\beta)) \longrightarrow \Map(\Gr_\bZ(\natural^{n^2}\Sigma^\circ(\beta))) = \Map(\Sigma(\alpha)),
\]
where each homomorphism is given by extending homeomorphisms by the identity. The result therefore follows by applying Proposition \ref{prop:grid-surface} to the surface $\natural^{n^2}\Sigma^\circ(\beta)$.
\end{proof}

\subsection{Limit ordinals.}

\begin{figure}[t]
    \centering
    \includegraphics[scale=0.7]{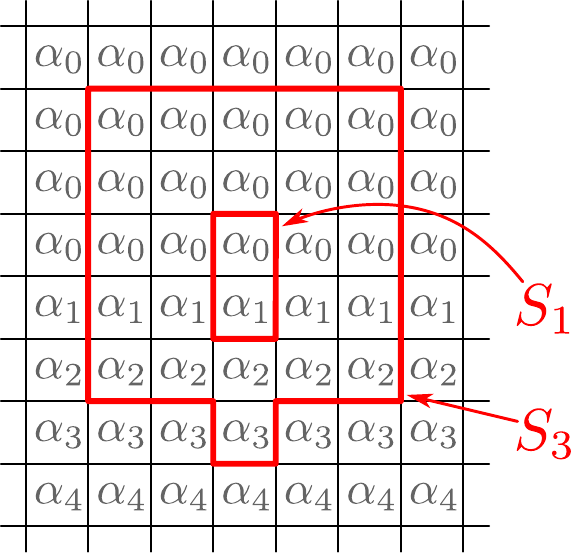}
    \caption{The surface $S \cong \Sigma(\lambda)$ from the proof of Proposition \ref{prop:omega-alpha-limit}, where each square labelled by $\alpha_n$ denotes a copy of $\Sigma^\circ(\alpha_n) = \bD^2 \smallsetminus [0,\omega^\alpha]$ (see Definition \ref{def:Sigma-alpha}). There is an exhaustive filtration of $S$ by properly-embedded subsurfaces $S_n \cong \Sigma^\circ(\alpha_n)$, described in the proof of Proposition \ref{prop:omega-alpha-limit}; the boundaries of $S_1$ and $S_3$ are outlined in red.}
    \label{fig:filtration-limit-ordinal}
\end{figure}

\begin{figure}[t]
    \centering
    \includegraphics[scale=0.7]{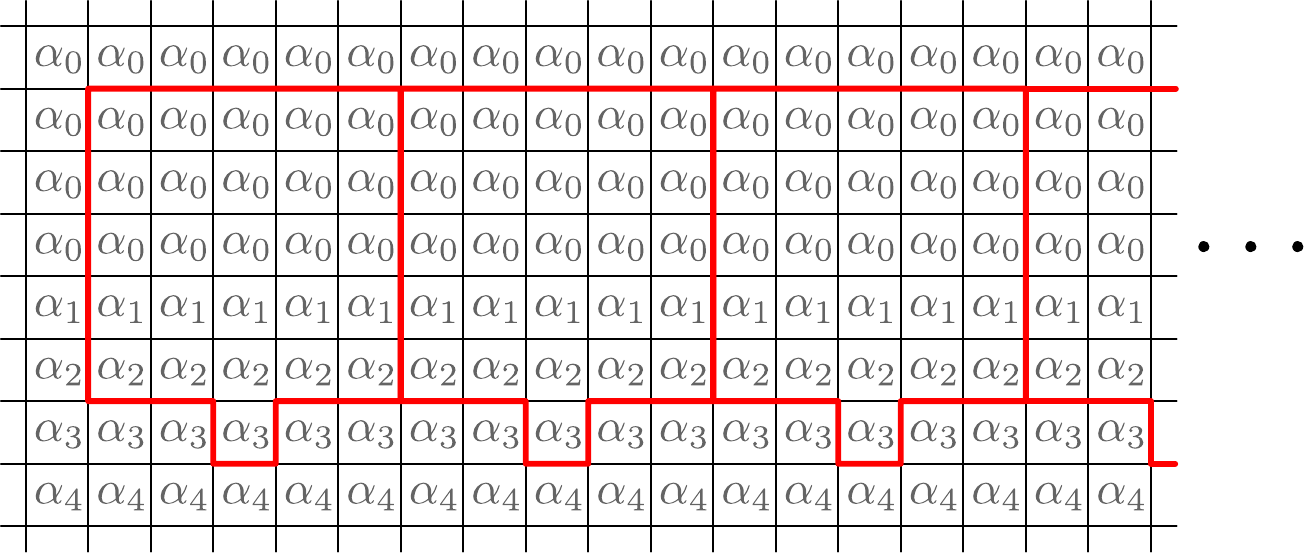}
    \caption{An extension of the inclusion $S_n \subset S$ (depicted in the case $n=3$) to a proper embedding $\Str(S_n) \hookrightarrow S$, proving that $S_n$ is shiftable in $S$.}
    \label{fig:filtration-limit-ordinal-shiftable}
\end{figure}

Let us now suppose that $\alpha = \lambda$ is a limit ordinal. Since it is also countable, its cofinality is precisely $\omega$ (see Remark \ref{rmk:cofinality}), meaning that there is a strictly increasing sequence $\alpha_n$ of ordinals (indexed by natural numbers $n \in \bN = \omega$) whose supremum is $\lambda$. Let us fix a choice of such a sequence for the remainder of this section.

\begin{prop}
\label{prop:omega-alpha-limit}
Let $\Sigma$ be a properly-embedded finite-type subsurface of $\Sigma(\lambda)$. Then $\Sigma$ is contained in a properly-embedded subsurface homeomorphic to $\Sigma^\circ(\alpha_n)$ for some $n \in \bN$. Moreover, this subsurface is shiftable in $\Sigma(\lambda)$.
\end{prop}

The second half of Theorem \ref{mainthm-genus-zero-infinite-punctures-2} -- the case when $\alpha = \lambda$ is a limit ordinal -- follows immediately:

\begin{cor}
\label{cor:omega-alpha-limit}
Suppose that $\lambda$ is a countable limit ordinal. Then the inclusion
\[
\iota_f(\Sigma(\lambda)) \colon \Map_f(\Sigma(\lambda)) \longhookrightarrow \Map(\Sigma(\lambda))
\]
induces the zero map on homology in positive degrees with any field coefficients.
\end{cor}
\begin{proof}
Let $\Sigma \subset \Sigma(\lambda)$ be a properly-embedded subsurface of finite type. Proposition \ref{prop:omega-alpha-limit} implies that the homomorphism $\Map(\Sigma) \to \Map(\Sigma(\lambda))$ factors as $\Map(\Sigma) \to \Map(\Sigma^\circ(\alpha_n)) \to \Map(\Sigma(\lambda))$, where the second homomorphism $\Map(\Sigma^\circ(\alpha_n)) \to \Map(\Sigma(\lambda))$ is induced by the inclusion of a shiftable subsurface; the result therefore follows from Corollary \ref{coro:shiftable}.
\end{proof}

\begin{proof}[Proof of Proposition \ref{prop:omega-alpha-limit}]
Let us construct a full-plane ``grid surface'' similarly to Definition \ref{def:grid-surface}, except that each square in the grid with coordinates $(i,j) \in \bZ \times \bZ$ is filled in with a copy of $\Sigma^\circ(\alpha_{\mathrm{max}(-j,0)})$. Equivalently, we begin with a copy of the half-plane grid surface $\Gr(\Sigma^\circ(\alpha_0))$, glue on a new row to the bottom of the grid filled with copies of $\Sigma^\circ(\alpha_1)$, then another row filled with copies of $\Sigma^\circ(\alpha_2)$, etc, until the whole grid is filled; see Figure \ref{fig:filtration-limit-ordinal}. Let us denote this surface by $S$. Also, for each $n \geq 1$, we denote by $S_n \subset S$ the subsurface given by the union of all $(2n-1)^2$ pieces whose coordinates $(i,j)$ satisfy $\mathrm{max}(\lvert i \rvert , \lvert j \rvert) \leq n-1$, together with the piece whose coordinates are $(0,-n)$; see Figure \ref{fig:filtration-limit-ordinal}. It follows from Lemma \ref{lem:ordinal-interval-calculations} that $S$ is homeomorphic to $\Sigma(\lambda)$ and each $S_n$ is homeomorphic to $\Sigma^\circ(\alpha_n)$. Also, the $S_n$ (for $n \in \bN$) form an exhaustive filtration of $S$ by properly-embedded subsurfaces.

Now let $\Sigma \subset \Sigma(\lambda) \cong S$ be any properly-embedded finite-type subsurface. It must be bounded away from the non-isolated end of $S$ ``at infinity'', so it must be contained in $S_n$ for some $n \in \bN$. To finish the proof, we just have to show that $S_n$ is shiftable in $S$: this is demonstrated pictorially in Figure \ref{fig:filtration-limit-ordinal-shiftable}.
\end{proof}

\section{The punctured and unpunctured Cantor tree surfaces}
\label{s:Cantor-tree-punctured}

In this section we prove Theorem \ref{mainthm-genus-zero-finite-punctures}\ref{mainthm-genus-zero-finite-punctures-1}, dealing with the case when $S$ has genus zero and has either $0$ or $1$ punctures (isolated planar ends). This corresponds to exactly two possible homeomorphism types of surfaces: the sphere minus a Cantor set $\bS^2 \smallsetminus \cC$ (the ``Cantor tree surface'') and the plane minus a Cantor set $\bR^2 \smallsetminus \cC$ (the ``punctured Cantor tree surface''). In both cases, we use the following result about the disc minus a Cantor set $\bD^2 \smallsetminus \cC$ (the ``one-holed Cantor tree surface'').

\begin{thm}[{\cite[Theorem~B]{PalmerWu2024}}]
\label{thm:PalmerWu2024}
$\Map(\bD^2 \smallsetminus \cC)$ is acyclic, i.e.~$H_i(\Map(\bD^2 \smallsetminus \cC)) = 0$ for all $i>0$.
\end{thm}

\begin{proof}[Proof of Theorem \ref{mainthm-genus-zero-finite-punctures}\ref{mainthm-genus-zero-finite-punctures-1}]
Let $S$ be an infinite-type surface of genus zero with either no punctures (isolated ends) or exactly one puncture; in other words $S$ is homeomorphic either to $\bS^2 \smallsetminus \cC$ or to $\bR^2 \smallsetminus \cC$. Let $\Sigma \subset S$ be a properly-embedded finite-type subsurface; our goal is to prove that $\Map(\Sigma) \to \Map(S)$ induces the zero map on homology in all positive degrees.

First assume that $\Sigma$ is compact. Since $S$ has genus zero, so does $\Sigma$, so it is homeomorphic to a sphere with $n$ holes for some $n\geq 1$. The complement $S \smallsetminus \Sigma$ thus has $n$ components, partitioning the end-space of $S$ into $n$ clopen subsets $E_1,\ldots,E_n$. Since the end-space of $S$ is homeomorphic either to $\cC$ or to $\cC \sqcup \{*\}$, and all non-empty clopen subsets of $\cC$ are homeomorphic to $\cC$ again, we may assume (reordering if necessary) that $E_1,\ldots,E_{n-1}$ are each homeomorphic to $\cC$ or $\varnothing$ and $E_n$ is homeomorphic to $\cC$ or $\varnothing$ or $\cC \sqcup \{*\}$ or $\{*\}$. Denote by $\Sigma'$ the subsurface of $S$ given by the union of $\Sigma$ together with the $n-1$ components of the complement $S \smallsetminus \Sigma$ corresponding to $E_1,\ldots,E_{n-1}$. Since $\Sigma'$ has genus zero, one (compact) boundary-component and has end-space homeomorphic to the disjoint union of some number (possibly zero) of copies of $\cC$, it is homeomorphic either to $\bD^2$ or to $\bD^2 \smallsetminus \cC$. The homomorphism $\Map(\Sigma) \to \Map(S)$ factors through $\Map(\Sigma')$, which is either the trivial group (if $\Sigma' \cong \bD^2$) or isomorphic to $\Map(\bD^2 \smallsetminus \cC)$, whose homology in all positive degrees vanishes by Theorem \ref{thm:PalmerWu2024}.

Now assume that $\Sigma$ is non-compact (but still finite-type). This implies that we must have $S \cong \bR^2 \smallsetminus \cC$ and $\Sigma \cong \Sigma_{0,n}^1$ for some $n\geq 1$. Removing an open annular neighbourhood of the unique puncture (point at infinity) of $S$, we obtain subsurfaces $S' \subset S$ and $\Sigma' \subset \Sigma$ such that $S' \cong \bD^2 \smallsetminus \cC$, $\Sigma' \cong \Sigma_{0,n+1}$ and $\Sigma' = \Sigma \cap S'$. Their mapping class groups fit into a map of central extensions
\begin{equation}
\label{eq:map-of-central-extensions}
\begin{tikzcd}
0 \ar[r] & \bZ \ar[rr] \ar[d,"\mathrm{id}"] && \Map(\Sigma') \ar[rr] \ar[d] && \Map(\Sigma) \ar[r] \ar[d] & 1 \\
0 \ar[r] & \bZ \ar[rr] && \Map(S') \ar[rr] && \Map(S) \ar[r] & 1.
\end{tikzcd}
\end{equation}
We claim that the central extension $0 \to \bZ \to \Map(\Sigma') \to \Map(\Sigma) \to 1$ on the top row is a trivial extension. Note that this will complete the proof, because it will then follow that the homomorphism $\Map(\Sigma) \to \Map(S)$ factors through $\Map(S') \cong \Map(\bD^2 \smallsetminus \cC)$, whose homology in all positive degrees vanishes by Theorem \ref{thm:PalmerWu2024}. The middle group is $\Map(\Sigma') \cong \Map(\Sigma_{0,n+1})$, which is the pure ribbon braid group on $n$ strands. It decomposes as $\bZ^n \times PB_n$, where $PB_n$ denotes the pure braid group on $n$ strands. By \cite[\S 9.3, p.~252]{FaMa11}, this decomposes as $PB_n \cong \cZ(PB_n) \times PB_n/\cZ(PB_n)$, where $\cZ(PB_n)$ denotes the centre of $PB_n$, which is infinite cyclic generated by the full twist $\Delta^2$. Putting this all together, we have a decomposition $\Map(\Sigma') \cong \bZ^n \times \bZ\{\Delta^2\} \times PB_n/\cZ(PB_n)$. The central subgroup $\bZ \subset \Map(\Sigma')$ under consideration is generated by the Dehn twist around the outer boundary, which corresponds under this identification to the element $((1,\ldots,1),\Delta^2) \in \bZ^n \times \bZ\{\Delta^2\}$. Via this description it is clear that it generates a direct factor of $\Map(\Sigma')$, in other words the quotient by this central subgroup admits a section: hence it is a trivial central extension.
\end{proof}

\begin{rem}
In contrast to our other vanishing results, Theorem \ref{mainthm-genus-zero-finite-punctures}\ref{mainthm-genus-zero-finite-punctures-1} holds for any coefficients, not only for coefficients in a field.
\end{rem}

\section{Non-trivial compactly-supported classes}
\label{s:non-trivial-classes}

In this section we prove all of our \emph{non-vanishing} results, except for Theorem \ref{mainthm-infinite-genus}\ref{mainthm-infinite-genus-3} whose proof is deferred to \S\ref{s:BT-construction}. In \S\ref{ss:finite-genus} we prove Theorem \ref{mainthm-finite-positive-genus}, dealing with the case when $S$ has non-zero but finite genus. In \S\ref{ss:finite-p-part-1} we prove Theorem \ref{mainthm-genus-zero-finite-punctures}\ref{mainthm-genus-zero-finite-punctures-2}, concerning Question \ref{q-finite-type} in the case when $S$ has finitely many but at least two punctures. In \S\ref{ss:topologically-distinguished-set-of-ends} we prove Theorem \ref{mainthm-genus-zero-infinite-punctures-1}, concerning the case when $S$ has genus zero and there is a finite, topologically-distinguished subset $A \subset \Ends(S)$ with $\lvert A \rvert \geq 4$. In the special case when $A$ is the set of punctures of $S$, this also proves Theorem \ref{mainthm-genus-zero-finite-punctures}\ref{mainthm-genus-zero-finite-punctures-3}.

\subsection{Finite, non-zero genus}
\label{ss:finite-genus}

In this subsection, we prove Theorem \ref{mainthm-finite-positive-genus}, which we state slightly more precisely as the following:

\begin{prop}
\label{prop:finite-genus}
Suppose that $1 \leq g_S < \infty$. Then the integral homology $H_*(\Map(S))$ contains non-zero classes that are supported on a compact subsurface of $S$ homeomorphic to $\Sigma_{g_S,1}$. More precisely, we may find such classes in degree $1$ when $g_S = 1$ and in degree $2$ when $g_S \geq 2$.
\end{prop}

In the proof, we will need the following calculations of low-degree homology groups of mapping class groups of closed, orientable surfaces.

\begin{lem}
\label{lem:low-degree-calculations}
We have $H_1(\Map(\Sigma_1)) \cong \bZ/12$ and
\[
H_2(\Map(\Sigma_g)) \cong \begin{cases}
\bZ/2 & g=2 \\
\bZ \oplus \bZ/2 & g=3 \\
\bZ & g\geq 4.
\end{cases}
\]
\end{lem}
\begin{proof}
For the first statement, $\Map(\Sigma_1)$ is isomorphic to $SL_2(\bZ)$, whose abelianisation is $\bZ/12$. For the second statement, see \cite[Theorem 6.1 and the paragraph following it]{Korkmaz02} for $g\geq 4$ and $g=2$. The case $g=3$ is not unambiguously settled in \cite{Korkmaz02}; instead, see \cite[Corollary 4.10]{Sakasai2012}. See also \cite[Lemma~A.1]{BCRR2020} for these and many more related calculations.
\end{proof}

\begin{proof}[Proof of Proposition \ref{prop:finite-genus}]
Since $S$ has finite genus $g_S$, all of its ends are planar and it is homeomorphic to $\Sigma_{g_S} \smallsetminus E$, where $E$ is the image of an embedding $\Ends(S) \hookrightarrow \Sigma_{g_S}$ of the end-space of $S$. Choose an embedded disc $D \subset \Sigma_{g_S}$ containing $E$ in its interior. There are homomorphisms
\begin{equation}
\label{eq:two-extensions}
\Map(\Sigma_{g_S} \smallsetminus \mathring{D}) \longrightarrow \Map(\Sigma_{g_S} \smallsetminus E) \longrightarrow \Map(\Sigma_{g_S})
\end{equation}
given respectively by extending homeomorphisms of $\Sigma_{g_S} \smallsetminus \mathring{D}$ (that are the identity on $\partial D$) by the identity on $D \smallsetminus E$ and extending homeomorphisms of $\Sigma_{g_S} \smallsetminus E$ to (its Freudenthal compactification) $\Sigma_{g_S}$ in the unique possible way (see~\cite[Appendix~B]{PalmerWu2} for why this determines a well-defined homomorphism of mapping class groups). The composition of the two homomorphisms \eqref{eq:two-extensions} is the classical \emph{capping map} $\Map(\Sigma_{g_S} \smallsetminus \mathring{D}) \to \Map(\Sigma_{g_S})$ that extends homeomorphisms by the identity on $D$. This map induces a surjection on $H_i$ whenever $g_S \geq \frac{3}{2}(i-1)$ by part \ref{thm:hmstb-capping} of Theorem \ref{thm:hmstb}. In particular, it induces a surjection on $H_1$ whenever $g_S \geq 1$ and on $H_2$ whenever $g_S \geq 2$. It therefore suffices to check that $H_1(\Map(\Sigma_1)) \neq 0$ and that $H_2(\Map(\Sigma_{g_S})) \neq 0$ when $g_S \geq 2$. This follows from Lemma \ref{lem:low-degree-calculations}.
\end{proof}

\subsection{Finitely many punctures but at least two}
\label{ss:finite-p-part-1}

In this subsection, we prove Theorem \ref{mainthm-genus-zero-finite-punctures}\ref{mainthm-genus-zero-finite-punctures-2}. In fact, this part of Theorem \ref{mainthm-genus-zero-finite-punctures} does not require the assumption that $g_S = 0$, so we may strengthen it to:

\begin{prop}
\label{prop:finite-p-part-1}
Suppose that $2 \leq p_S < \infty$. Then $H_*(\Map(S))$ contains a non-trivial class that is supported on a properly-embedded finite-type subsurface of $S$.
\end{prop}
\begin{proof}
Since $p = p_S$ is finite, there is a properly-embedded subsurface of $S$ homeomorphic to the punctured disc $\bD^2 \smallsetminus P$, where $P$ is a finite set of size $p$ in the interior of $\bD^2$. This induces an extension map $B_p = \Map(\bD^2 \smallsetminus P) \to \Map(S)$, where $B_p$ denotes the braid group on $p$ strands. For any homeomorphism of $S$, its induced action on the end-space $\Ends(S)$ must send the subset of punctures onto itself, so there is a well-defined map $\Map(S) \to \mathfrak{S}_p$ recording this permutation. The composition $B_p \to \mathfrak{S}_p$ records the permutation induced by a braid, and is surjective. Since abelianisation $(-)^{ab} = H_1(-)$ is a right-exact functor, the composition of the induced maps $H_1(B_p) \to H_1(\Map(S)) \to H_1(\mathfrak{S}_p)$ is also surjective. Since $H_1(\mathfrak{S}_p) \cong \bZ/2$ (here we are using the assumption that $p\geq 2$), we may choose a lift $\alpha \in H_1(B_p)$ of the non-trivial element of $H_1(\mathfrak{S}_p)$. The image of $\alpha$ in $H_1(\Map(S))$ is then a non-trivial class supported on a properly-embedded finite-type subsurface.
\end{proof}

The above proof does not work when $p = p_S = 1$, since $H_1(\mathfrak{S}_p)$ is trivial in this case. Indeed, in this case, the answer to Question \ref{q-finite-type} depends also on the genus of $S$. If $g_S = 0$ then $S$ must be homeomorphic to $\bR^2 \smallsetminus \cC$ and the answer is given in \S\ref{s:Cantor-tree-punctured} above. The case when $0 < g_S < \infty$ is covered by Proposition \ref{prop:finite-genus} above. The case when $g_S = \infty$ is dealt with in \S\ref{s:BT-construction} below.

\subsection{Genus zero with a finite, topologically-distinguished set of ends}
\label{ss:topologically-distinguished-set-of-ends}

We next prove Theorem~\ref{mainthm-genus-zero-infinite-punctures-1} (which in particular implies Theorem \ref{mainthm-genus-zero-finite-punctures}\ref{mainthm-genus-zero-finite-punctures-3}). Recall from Definition \ref{def:topologically-distinguished-subset} the notion of a \emph{topologically distinguished subset}. The following is a refinement of Theorem~\ref{mainthm-genus-zero-infinite-punctures-1}.

\begin{prop}
\label{prop:at-least-3-distinguished-points}
Suppose that $S$ has genus zero and that $\Ends(S)$ has a finite, topologically distinguished subset of size $n\geq 2$. Set $k=n-1$ if $n$ is even and $k = \frac12(n-1)$ if $n$ is odd. Then there is a compact subsurface $\Sigma \subset S$ and a commutative diagram
\begin{equation}
\label{eq:at-least-3-distinguished-points}
\begin{tikzcd}
H_1(\Map(\Sigma)) \ar[rr,two heads] \ar[d] && \bZ/k \ar[d,"\cdot 2"] \\
H_1(\Map(S)) \ar[rr] && \bZ/2k,
\end{tikzcd}
\end{equation}
where the left-hand vertical map is induced by $\Map(\Sigma) \subset \Map(S)$, the right-hand vertical map is multiplication by $2$ and the top horizontal map is surjective. In particular, if $n\geq 4$, there are non-trivial classes in $H_1(\Map(S))$ that are supported on the compact subsurface $\Sigma \subset S$.
\end{prop}

\begin{rem}
We do not require the finite, topologically-distinguished subset of $\Ends(S)$ in Proposition \ref{prop:at-least-3-distinguished-points} to be homogeneous; for example, it may consist of $n$ points that are each (individually) topologically-distinguished. The bottom horizontal map of \eqref{eq:at-least-3-distinguished-points} is surjective if and only if this is \emph{not} the case, i.e.~two of the points of the (chosen) topologically-distinguished subset are \emph{similar}, i.e.~have homeomorphic open neighbourhoods.
\end{rem}

\begin{proof}[Proof of Proposition \ref{prop:at-least-3-distinguished-points}]
We first note that the second statement follows from the first: when $n\geq 4$ we have $k\geq 2$, so the element $2 \in \bZ/2k$ is non-trivial and pulls back through $H_1(\Map(S))$ to $H_1(\Map(\Sigma))$. Hence we just have to prove the first statement.

Denote by $A \subset \Ends(S)$ a topologically distinguished subset of size $n$. Since $\Ends(S)$ is Hausdorff and zero-dimensional, we may partition it into clopen subsets $E_1,\ldots,E_n$ such that each $E_i$ contains exactly one point of $A$. Let $S_i = \bD^2 \smallsetminus E_i$ for $i\in\{1,\ldots,n\}$ and denote by $\Sigma_{0,n}$ the compact, connected, genus-$0$ surface with $n$ boundary components. Gluing $S_1,\ldots,S_n$ into the $n$ holes of $\Sigma_{0,n}$ we obtain $\bS^2 \smallsetminus \Ends(S)$, which is homeomorphic to $S$. There is therefore an extension homomorphism $\Map(\Sigma_{0,n}) \to \Map(S)$ given by extending homeomorphisms by the identity on each $S_i$. On the other hand, since $A \subset \Ends(S)$ is a topologically distinguished subset, we have a homomorphism $\Map(S) \to \Map(\bS^2 \smallsetminus A)$ given by filling in all ends of $S$ except $A$. Next, there is a central extension \cite[\S 9.1.4]{FaMa11}
\[
1 \to \bZ/2 \longrightarrow B_n(\bS^2) \longrightarrow \Map(\bS^2 \smallsetminus A) \to 1,
\]
where the generator of the kernel is sent to a full twist in the spherical braid group $B_n(\bS^2)$. This is sent to $n(n-1)$ in $B_n(\bS^2)^{ab} \cong \bZ/(2n-2)$, which is $0$ if $n$ is even and $n-1$ if $n$ is odd. Let us consider the quotient of $B_n(\bS^2)$ onto its abelianisation $\bZ/(2n-2)$ when $n$ is even and the further quotient onto $\bZ/(n-1)$ when $n$ is odd; we may write this uniformly as the quotient $B_n(\bS^2) \twoheadrightarrow \bZ/2k$ where $k=n-1$ if $n$ is even and $k = \frac12(n-1)$ if $n$ is odd. By construction, the kernel $\bZ/2 \subset B_n(\bS^2)$ of the central extension above is sent to zero in this quotient, so it factors through a quotient $\Map(\bS^2 \smallsetminus A) \twoheadrightarrow \bZ/2k$. Putting everything together, we have maps
\begin{equation}
\label{eq:composition-onto-Zk}
\Map(\Sigma_{0,n}) \longrightarrow \Map(S) \longrightarrow \Map(\bS^2 \smallsetminus A) \longrightarrow \bZ/2k.
\end{equation}
The composition $\Map(\Sigma_{0,n}) \to \bZ/2k$ is not surjective: its image is instead the cyclic subgroup of order $k$ generated by $2 \in \bZ/2k$. To see this, first note that a pre-image of the element $2 \in \bZ/2k$ in $\Map(\Sigma_{0,n})$ is given by a homeomorphism that ``pushes'' one boundary component in a full loop around another boundary component. This implies that the image contains the cyclic subgroup generated by $2$. On the other hand, it cannot be larger than this, since each element of $\Map(\Sigma_{0,n})$ fixes the $n$ boundary components pointwise and hence its induced permutation of $A$ is trivial, which is an \emph{even} permutation. Thus we have the following commutative diagram:
\[
\begin{tikzcd}
\Map(\Sigma_{0,n}) \ar[rr,two heads] \ar[d] && \bZ/k \ar[d,"\cdot 2"] \\
\Map(S) \ar[rr] && \bZ/2k.
\end{tikzcd}
\]
Taking abelianisations, we obtain the desired commutative diagram \eqref{eq:at-least-3-distinguished-points} with $\Sigma = \Sigma_{0,n}$.
\end{proof}

\section{Classes detected by wreath products of the circle group}
\label{s:BT-construction}

In this section we prove Theorem \ref{mainthm-infinite-genus}\ref{mainthm-infinite-genus-3}, which we restate in a stronger form as Proposition \ref{prop:finite-p-part-2} below. We want to consider surfaces of infinite genus with finitely many (and at least one) punctures. However, it will be more convenient to think of the punctures as marked points, so we fix a surface $S$ of infinite genus and \emph{no} punctures, together with a non-empty, finite subset $P \subset S$, and we shall be interested in Question \ref{q-finite-type-pure} for the surface $S \smallsetminus P$. In other words, we are interested in the image of the map\begin{equation}
\label{eq:inclusion-f}
H_*(\Map_c(S \smallsetminus P)) \longrightarrow H_*(\Map(S \smallsetminus P))
\end{equation}
induced by the inclusion $\Map_c(S \smallsetminus P) \subset \Map(S \smallsetminus P)$. Question \ref{q-finite-type-pure} asks whether the image of the map \eqref{eq:inclusion-f} is non-zero for some positive degree $* > 0$. In fact we may prove that it is non-zero in \emph{every even degree}:

\begin{prop}
\label{prop:finite-p-part-2}
Let $S$ be a connected, orientable surface of infinite genus with no punctures and $P \subset S$ a non-empty, finite subset. Then the image of the map \eqref{eq:inclusion-f} contains a $\bZ$ summand in every even degree; in particular it is non-zero.
\end{prop}

A key ingredient of the proof is a construction due to Bödigheimer and Tillmann \cite{BoedigheimerTillmann2001}.

\begin{notation}
For a surface $S$ (possibly with boundary) and finite subset $P \subset \mathring{S}$ of its interior, denote by $\Diff(S,P)$ the topological group of diffeomorphisms of $S$ that fix $P$ setwise and $\partial S$ pointwise, equipped with the weak (``smooth compact-open'') topology \cite[\S 2.1]{Hirsch1976}.
\end{notation}

\begin{defn}[{\cite[\S 3]{BoedigheimerTillmann2001}}]
\label{def:map-to-wreath-product}
For a surface $S$ (possibly with boundary) and finite subset $P \subset \mathring{S}$ of size $p = \lvert P \rvert$, let
\begin{equation}
\label{eq:map-to-wreath-product}
\tau \colon \Diff(S,P) \longrightarrow \bS^1 \wr \mathfrak{S}_p = (\bS^1)^p \rtimes \mathfrak{S}_p
\end{equation}
be the continuous homomorphism defined as follows. Choose a bijection $i \mapsto x_i \colon \{1,\ldots,p\} \to P$ and choose a tangent vector $v_i \in T_{x_i}S$ for each $i \in \{1,\ldots,p\}$. For a diffeomorphism $\varphi$, the map $\tau$ records the induced permutation $\sigma(\varphi)$ of $\{1,\ldots,p\}$ under the chosen bijection and the angle $\theta_i(\varphi)$ between $D\varphi(v_i)$ and $v_{\sigma(i)}$ for each $i \in \{1,\ldots,p\}$, where $D\varphi$ denotes the derivative of $\varphi$.
\end{defn}

\begin{notation}
Write $BG$ for the classifying space of a topological group $G$. When $G$ is discrete the homology of $BG$ agrees with the group homology of $G$, so we shall write $H_*(BG) = H_*(G)$.
\end{notation}

\begin{construction}
Let $S$ be a surface with no punctures and $P \subset S$ a finite subset of size $p = \lvert P \rvert$. We shall use \eqref{eq:map-to-wreath-product} to construct a map
\begin{equation}
\label{eq:BT-map}
B\Map_f(S \smallsetminus P) \longrightarrow B(\bS^1 \wr \mathfrak{S}_p).
\end{equation}
Let $\Sigma \subset S$ be a compact subsurface containing $P$ in its interior. We then have maps
\begin{equation}
\label{eq:zig-zag}
\Map(\Sigma \smallsetminus P) \cong \pi_0(\Diff(\Sigma,P)) \longtwoheadleftarrow \Diff(\Sigma,P) \longrightarrow \bS^1 \wr \mathfrak{S}_p,
\end{equation}
where the right-hand map is \eqref{eq:map-to-wreath-product}. These are all compatible with the maps induced by inclusions of subsurfaces, so there are induced maps of colimits
\begin{equation}
\label{eq:zig-zag-colimit}
\underset{\Sigma}{\mathrm{colim}}(\Map(\Sigma \smallsetminus P)) \cong \underset{\Sigma}{\mathrm{colim}}(\pi_0(\Diff(\Sigma,P))) \longtwoheadleftarrow \underset{\Sigma}{\mathrm{colim}}(\Diff(\Sigma,P)) \longrightarrow \bS^1 \wr \mathfrak{S}_p,
\end{equation}
where each colimit is taken over the poset of all compact subsurfaces $\Sigma \subset S$ with $P \subset \mathring{\Sigma}$. Since $S$ has no punctures, the subsurfaces $\Sigma \smallsetminus P \subset S \smallsetminus P$ form a cofinal family in $\fF(S \smallsetminus P)$, so the left-hand group in \eqref{eq:zig-zag-colimit} may be identified with $\Map_f(S \smallsetminus P)$, by Lemma \ref{lem:colimits}.

The middle map in \eqref{eq:zig-zag} is a weak homotopy equivalence by \cite{EarleEells1969,EarleSchatz1970} and hence so is the middle map in \eqref{eq:zig-zag-colimit}. Taking classifying spaces and inverting this map, we obtain the desired map \eqref{eq:BT-map}.
\end{construction}

\begin{rem}
By construction, restricting \eqref{eq:BT-map} to (the classifying space of) $\Map_c(S \smallsetminus P)$, we obtain a commutative square:
\begin{equation}
\label{eq:BT-map-restriction}
\begin{tikzcd}
B\Map_c(S \smallsetminus P) \ar[rr] \ar[d] && B(\bS^1)^p \ar[d] \\
B\Map_f(S \smallsetminus P) \ar[rr] && B(\bS^1 \wr \mathfrak{S}_p).
\end{tikzcd}
\end{equation}
\end{rem}

The key ingredient to prove Proposition \ref{prop:finite-p-part-2} is the following corollary of the main result of \cite{BoedigheimerTillmann2001}.

\begin{thm}
\label{thm:BT}
If $S$ is a connected, orientable surface of infinite genus with no punctures and $P \subset S$ is a finite subset, then the maps
\begin{equation}
\label{eq:BT-map-on-homology}
H_*(\Map_f(S \smallsetminus P)) \longrightarrow H_*(B(\bS^1 \wr \mathfrak{S}_p)) \qquad\text{and}\qquad H_*(\Map_c(S \smallsetminus P)) \longrightarrow H_*(B(\bS^1)^p)
\end{equation}
induced on homology by \eqref{eq:BT-map-restriction} each admit a section. Moreover, these sections are compatible in the sense that if we consider the commutative square induced on homology by \eqref{eq:BT-map-restriction} and replace its horizontal maps with these two sections, then the result is still a commutative square.
\end{thm}
\begin{proof}
In fact, \cite{BoedigheimerTillmann2001} makes a stronger statement. Let us write $(-)^+$ for the \emph{Quillen plus-construction} of a topological space (see \cite[\S 2]{BoedigheimerTillmann2001} for a brief summary of some key properties of the plus-construction and for example \cite{Berrick1982} for further details). According to \cite[Theorem~1.1~(2)]{BoedigheimerTillmann2001}, the space $B\Map_f(S \smallsetminus P)^+$ splits, up to homotopy equivalence, as the product of $B\Gamma_{\infty}^+$ and $B(\bS^1 \wr \mathfrak{S}_p)^+$, where $\Gamma_\infty$ denotes the colimit of $\Map(\Sigma_{g,1})$ as $g\to\infty$, and (the plus-construction of) the map \eqref{eq:BT-map} is the projection onto the second factor of this decomposition. It therefore admits a section up to homotopy, so the result follows upon taking homology since $(-)^+$ does not change the homology of a space. In \cite{BoedigheimerTillmann2001}, this result is stated for the particular surface $S = \mathrm{colim}_{g\to\infty}(\Sigma_{g,1})$, but the homology $H_*(\Map_f(S \smallsetminus P))$ is the same for any surface $S$ satisfying the hypotheses of the theorem, by \cite{Har85}; alternatively, the proof of \cite{BoedigheimerTillmann2001} goes through for any such surface $S$, by taking the colimit of an appropriate diagram of stabilisation maps, corresponding to a filtration of $S$ by compact subsurfaces.

This deals with the left-hand map of \eqref{eq:BT-map-on-homology}. Restricting to $\Map_c(S \smallsetminus P) \subseteq \Map_f(S \smallsetminus P)$ corresponds to restricting to the \emph{pure} mapping class group for each finite-type subsurface over which we are taking the colimit (see Lemma \ref{lem:colimits}). Hence the exact same argument also proves that the right-hand map of \eqref{eq:BT-map-on-homology} admits a section, using \cite[Theorem~1.1~(1)]{BoedigheimerTillmann2001} instead of \cite[Theorem~1.1~(2)]{BoedigheimerTillmann2001}.

Finally, to see that these two sections are compatible in the sense described, we first note that the homotopy splittings \cite[Theorem~1.1~(1)]{BoedigheimerTillmann2001} and \cite[Theorem~1.1~(2)]{BoedigheimerTillmann2001} are both special cases of the homotopy splitting \cite[Theorem~3.1]{BoedigheimerTillmann2001}. The latter depends on a choice of subgroup of the symmetric group $\mathfrak{S}_k$ (where $k$ denotes the cardinality of the finite set $P$); then \cite[Theorem~1.1~(1)]{BoedigheimerTillmann2001} and \cite[Theorem~1.1~(2)]{BoedigheimerTillmann2001} correspond to the trivial subgroup $\{1\}$ and the whole group $\mathfrak{S}_k$ respectively. Now it is clear from the proof of \cite[\S 3]{BoedigheimerTillmann2001} that the splitting of \cite[Theorem~3.1]{BoedigheimerTillmann2001} is natural with respect to the lattice of subgroups of $\mathfrak{S}_k$. In particular, naturality with respect to $\{1\} \subset \mathfrak{S}_k$ implies the desired compatibility statement.
\end{proof}

To complete the proof of Proposition \ref{prop:finite-p-part-2} we need one further ingredient.

\begin{prop}
\label{prop:extension}
If $S$ is a connected, orientable surface with no punctures and $P \subset S$ is a finite subset, then the map \eqref{eq:BT-map} extends along $B(\mathrm{incl}) \colon B\Map_f(S \smallsetminus P) \to B\Map(S \smallsetminus P)$.
\end{prop}

Before proving this, we first explain how (together with Theorem \ref{thm:BT}) it implies Proposition \ref{prop:finite-p-part-2}, and hence Theorem \ref{mainthm-infinite-genus}\ref{mainthm-infinite-genus-3}.

\begin{proof}[Proof of Proposition \ref{prop:finite-p-part-2}]
By Theorem \ref{thm:BT} and Proposition \ref{prop:extension}, we have a commutative diagram
\[
\begin{tikzcd}[column sep=small]
H_*(B\bS^1) \ar[r] & H_*(B(\bS^1)^p) \ar[r] \ar[d] & H_*(\Map_c(S \smallsetminus P)) \ar[rr] \ar[d] && H_*(B(\bS^1)^p) \ar[d] \\
& H_*(B(\bS^1 \wr \mathfrak{S}_p)) \ar[r] & H_*(\Map_f(S \smallsetminus P)) \ar[rr] \ar[dr] && H_*(B(\bS^1 \wr \mathfrak{S}_p)) \ar[d] \\
&&& H_*(\Map(S \smallsetminus P)) \ar[ur] & H_*(B\bS^1)
\end{tikzcd}
\]
in which the horizontal compositions $H_*(B(\bS^1)^p) \to H_*(B(\bS^1)^p)$ and $H_*(B(\bS^1 \wr \mathfrak{S}_p)) \to H_*(B(\bS^1 \wr \mathfrak{S}_p))$ are identities. The top-left horizontal map is induced by the homomorphism $\bS^1 \to (\bS^1)^p$ sending $t \mapsto (t,0,\ldots,0)$ and the bottom-right vertical map is induced by the homomorphism $\bS^1 \wr \mathfrak{S}_p \to \bS^1$ sending $(t_1,\ldots,t_p;\sigma) \mapsto t_1 + \cdots + t_p$. By construction, the map from the top-left to the bottom-right of the diagram is the identity map. Thus we have factored the identity map of $H_*(B\bS^1) = H_*(\bC\bP^\infty)$ through the map \eqref{eq:inclusion-f}. It follows that the image of \eqref{eq:inclusion-f} contains a direct summand isomorphic to $H_*(\bC\bP^\infty)$, which is a copy of $\bZ$ in every even degree.
\end{proof}

\begin{proof}[Proof of Proposition \ref{prop:extension}]
At the level of diffeomorphism groups, the construction clearly extends to a well-defined continuous homomorphism $\Diff(S,P) \to \bS^1 \wr \mathfrak{S}_p$. Indeed, Definition \ref{def:map-to-wreath-product} does not make any compactness assumptions. The only subtlety lies in descending this homomorphism to the mapping class group.

We first recall the construction of the map \eqref{eq:BT-map} in more detail. In the diagram in Figure \ref{fig:BT-diagram}, $S$ and $P \subset S$ are as in Proposition \ref{prop:extension} and $\Sigma \subset S$ is any compact subsurface containing $P$ in its interior. The natural map from the diffeomorphism group to the homeomorphism group of any smooth surface is a weak equivalence (in particular an isomorphism on $\pi_0$)\footnote{This follows essentially from smoothing theory \cite[Essay~V]{KirbySiebenmann}. See \cite[Appendix~A]{PalmerWu2024} for a brief explanation, which emphasises that the underlying surface does not have to be compact.} and the restriction map $\Homeo(S,P) \to \Homeo(S \smallsetminus P)$ is an isomorphism of topological groups when $S$ has no punctures, since one may define an inverse by extending homeomorphisms uniquely to the Freudenthal compactification of $S \smallsetminus P$ and then discarding all ends that are not punctures (these are preserved by any homeomorphism). Thus all of the vertical maps in Figure \ref{fig:BT-diagram} are either weak equivalences or isomorphisms. The horizontal map across the top is the homomorphism \eqref{eq:map-to-wreath-product} of Definition \ref{def:map-to-wreath-product}. To identify its domain with $\Map(\Sigma \smallsetminus P) = \pi_0(\Homeo(\Sigma \smallsetminus P))$, we need to know that the diagonal maps on the left-hand side are also weak equivalences, which follows either from \cite{EarleEells1969,EarleSchatz1970} at the level of diffeomorphism groups or from \cite{Hamstrom1966} at the level of homeomorphism groups. Taking the colimit over all $\Sigma$, and then taking classifying spaces, we obtain the map \eqref{eq:BT-map}.

To see that \eqref{eq:BT-map} extends to $B\Map(S \smallsetminus P)$, we need to know that the diagonal maps on the right-hand side are also weak equivalences. This follows from the main result of \cite{Yagasaki2000}, which extends \cite{Hamstrom1966} to non-compact surfaces, proving that $\Homeo(S \smallsetminus P)$ has contractible components (and thus contractible path-components), so the projection $\Homeo(S \smallsetminus P) \to \pi_0(\Homeo(S \smallsetminus P))$ is a weak equivalence. There is a small additional subtlety: for this projection map to make sense, one has to equip $\Map(S \smallsetminus P) = \pi_0(\Homeo(S \smallsetminus P))$ with the quotient topology induced by the compact-open topology, whereas we are interested in it as an abstract group, equivalently equipped with the discrete topology. Let us temporarily take the convention that $\Map(S \smallsetminus P)$ denotes the mapping class group with the quotient topology and $\Map(S \smallsetminus P)^\delta$ denotes the same group with the discrete topology. Since $\Map(S \smallsetminus P)$ is totally disconnected (in fact it is homeomorphic to the Baire space $\bN^\bN$ \cite[Thm~4.2]{AV20}), the map $\Map(S \smallsetminus P)^\delta \to \Map(S \smallsetminus P)$ given by the identity of the underlying groups is a weak equivalence. Together, this implies that the map \eqref{eq:BT-map} extends from $B\Map_f(S \smallsetminus P) = B\underset{\Sigma}{\mathrm{colim}}(\Map(\Sigma \smallsetminus P))$ to $B\Map(S \smallsetminus P)^\delta$.
\end{proof}

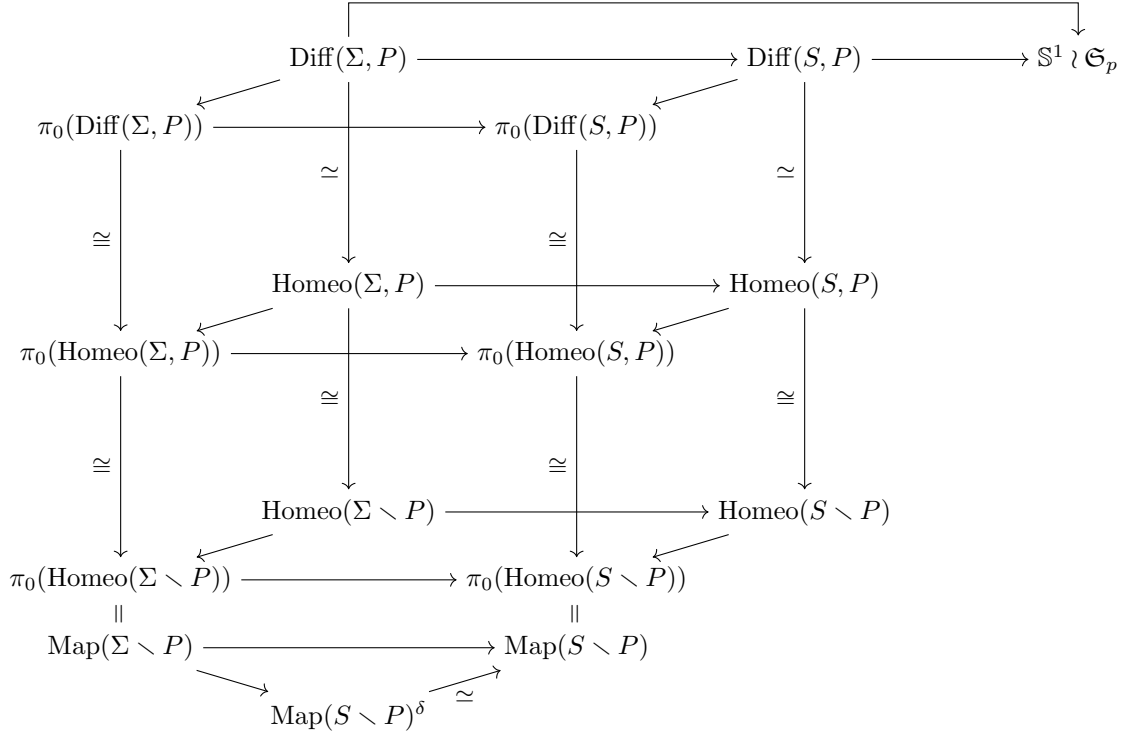
\begin{figure}[t]
\centering
\begin{equation*}
\begin{split}
\begin{tikzpicture}
[x=1.2mm,y=1.5mm]
\node (z) at (80,40) {$\bS^1 \wr \mathfrak{S}_p$};
\node (tl) at (0,40) {$\Diff(\Sigma,P)$};
\node (tr) at (50,40) {$\Diff(S,P)$};
\node (ml) at (0,20) {$\Homeo(\Sigma,P)$};
\node (mr) at (50,20) {$\Homeo(S,P)$};
\node (bl) at (0,0) {$\Homeo(\Sigma \smallsetminus P)$};
\node (br) at (50,0) {$\Homeo(S \smallsetminus P)$};
\node (tl0) at (-25,34) {$\pi_0(\Diff(\Sigma,P))$};
\node (tr0) at (25,34) {$\pi_0(\Diff(S,P))$};
\node (ml0) at (-25,14) {$\pi_0(\Homeo(\Sigma,P))$};
\node (mr0) at (25,14) {$\pi_0(\Homeo(S,P))$};
\node (bl0) at (-25,-6) {$\pi_0(\Homeo(\Sigma \smallsetminus P))$};
\node (br0) at (25,-6) {$\pi_0(\Homeo(S \smallsetminus P))$};
\node (blM) at (-25,-12) {$\Map(\Sigma \smallsetminus P)$};
\node (bmM) at (0,-18) {$\Map(S \smallsetminus P)^\delta$};
\node (brM) at (25,-12) {$\Map(S \smallsetminus P)$};
\draw[->] (tl) -- (0,45) -- (80,45) -- (z);
\draw[->] (tr) to (z);
\draw[->] (tl) to (tr);
\draw[->] (ml) to (mr);
\draw[->] (bl) to (br);
\draw[->] (tl0) to (tr0);
\draw[->] (ml0) to (mr0);
\draw[->] (bl0) to (br0);
\draw[->] (tl) to node[left,font=\small]{$\simeq$} (ml);
\draw[->] (ml) to node[left,font=\small]{$\cong$} (bl);
\draw[->] (tr) to node[left,font=\small]{$\simeq$} (mr);
\draw[->] (mr) to node[left,font=\small]{$\cong$} (br);
\draw[->] (tl0) to node[left,font=\small]{$\cong$} (ml0);
\draw[->] (ml0) to node[left,font=\small]{$\cong$} (bl0);
\draw[->] (tr0) to node[left,font=\small]{$\cong$} (mr0);
\draw[->] (mr0) to node[left,font=\small]{$\cong$} (br0);
\draw[->] (tl) to (tl0);
\draw[->] (tr) to (tr0);
\draw[->] (ml) to (ml0);
\draw[->] (mr) to (mr0);
\draw[->] (bl) to (bl0);
\draw[->] (br) to (br0);
\draw[->] (blM) to (brM);
\draw[->] (blM) to (bmM);
\draw[->] (bmM) to node[below,font=\small]{$\simeq$} (brM);
\node at (-25,-9) {\rotatebox{90}{$=$}};
\node at (25,-9) {\rotatebox{90}{$=$}};
\end{tikzpicture}
\end{split}
\end{equation*}
\caption{The diagram used to construct the extension of \eqref{eq:BT-map} to $B\Map(S \smallsetminus P)$ in the proof of Proposition \ref{prop:extension}. The surface $S$ is connected, orientable and has no punctures (in other words, all of its ends are either non-planar or non-isolated), $P \subset S$ is a finite subset and $\Sigma \subset S$ is a compact subsurface containing $P$ in its interior.}
\label{fig:BT-diagram}
\end{figure}

\bibliographystyle{alpha}
\bibliography{references.bib}

\end{document}